\documentclass[12pt]{amsart}
\usepackage{times}
\DeclareMathAlphabet{\mathpzc}{OT1}{pzc}{m}{it}

\usepackage[T1]{fontenc}
\usepackage{dsfont}
\usepackage{mathrsfs}
\usepackage[colorlinks]{hyperref}
\usepackage{xcolor}
\usepackage[a4paper,asymmetric]{geometry}
\usepackage{mathscinet}
\usepackage{fullpage}
\usepackage{latexsym}
\usepackage{graphicx}
\usepackage{epstopdf}
\usepackage{amsthm}
\usepackage{amssymb}
\usepackage{amsfonts}
\usepackage{amsmath}

\newtheorem{theorem}{Theorem}[section]
\newtheorem{lemma}[theorem]{Lemma}
\newtheorem{proposition}[theorem]{Proposition}

\theoremstyle{definition}

\newtheorem{remark}[theorem]{Remark}

\newtheorem*{assA*}{Assumption A}

\numberwithin{equation}{section}

\newcommand{\Ee}{\mathds{E}}
\newcommand{\real}{\mathds{R}}
\newcommand{\rd}{\mathds{R}^d}

\newcommand\II{\mathds{1}}
\newcommand{\dif}{\mathrm{d}}
\newcommand\iup{\mathrm{i}}
\newcommand\eup{\mathrm{e}}

\begin{document}

\title[Approximation of the invariant measure of stable SDEs by an EM scheme]
{Approximation of the invariant measure of stable SDEs by an Euler--Maruyama scheme}

\author[P.~Chen]{Peng Chen}
\address[P.~Chen]{School of Mathematics, Nanjing University of Aeronautics and
Astronautics, Nanjing 211106, China}
\email{chenpengmath@nuaa.edu.cn}

\author[C.-S.~Deng]{Chang-Song Deng}
\address[C.-S.~Deng]{School of Mathematics and Statistics, Wuhan University, Wuhan 430072, China}
\email{dengcs@whu.edu.cn}

\author[R.L.~Schilling]{Ren\'{e} L.\ Schilling}
\address[R.L.\ Schilling]{%
TU Dresden, Fakult\"{a}t Mathematik, Institut f\"{u}r Ma\-the\-ma\-ti\-sche Sto\-cha\-stik, 01062 Dresden, Germany}
\email{rene.schilling@tu-dresden.de}

\author[L.~Xu]{Lihu Xu}
\address[L.~Xu]{1. Department of Mathematics, Faculty of Science and Technology, University of Macau, Macau S.A.R., China. 2. Zhuhai UM Science \& Technology Research Institute, Zhuhai, China}
\email{lihuxu@um.edu.mo}


\makeatletter
\@namedef{subjclassname@2020}{%
  \textup{2020} Mathematics Subject Classification}
\makeatother


\begin{abstract}
    We propose two Euler-Maruyama (EM) type numerical schemes in order to approximate the invariant measure of a stochastic differential equation (SDE) driven by an $\alpha$-stable L\'evy process ($1<\alpha<2$): an approximation scheme with the $\alpha$-stable distributed noise
    and a further scheme with Pareto-distributed noise. Using a discrete version of Duhamel's principle and Bismut's formula in Malliavin calculus, we prove that the error bounds in Wasserstein-$1$ distance are in the order of $\eta^{1-\epsilon}$ and $\eta^{\frac2{\alpha}-1}$, respectively, where $\epsilon \in (0,1)$ is arbitrary and $\eta$ is the step size of the approximation schemes. For the Pareto-driven scheme, an explicit calculation for Ornstein--Uhlenbeck $\alpha$-stable process shows that the rate $\eta^{\frac2{\alpha}-1}$ cannot be improved.
\end{abstract}

\maketitle

\section{Introduction}
We study the solution $(X_t)_{t \ge 0}$ of the following stochastic differential equation (SDE) driven by an $\alpha$-stable L\'evy process:
\begin{align}\label{SDE}
    \dif X_t = b(X_t)\,\dif t+\dif Z_{t},\quad X_0=x,
\end{align}
where $x \in \rd$ is the starting point, $(Z_t)_{t \ge 0}$ is a $d$-dimensional, rotationally invariant $\alpha$-stable L\'evy process with index $\alpha \in (1,2)$, and $b: \rd \to \rd$ is a function satisfying \textbf{Assumption A} below.

The Euler-Maruyama (EM) scheme of the SDE \eqref{SDE}, with a step size  $\eta\in(0,1)$, is defined by
\begin{align}\label{EM2}
    Y_0 = x,\quad Y_{k+1} = Y_{k} + \eta b(Y_{k})+(Z_{(k+1)\eta}-Z_{k\eta}), \quad k=0,1,2,\dots,
\end{align}
see, e.g.\ \cite{TA18,KS19}. It is easy to see that $(Y_k)_{k \ge 0}$ is a Markov chain. A drawback of the scheme \eqref{EM2} is that there is no explicit representation for the probability density of $\alpha$-stable noise $Z_{(k+1)\eta}-Z_{k\eta}$, $\alpha\in(1,2)$, making the numerical simulation is complicated and numerically expensive, see the very recent monograph \cite[Section 1.9]{Nolbook} for a detailed discussion about the difficulties arising in the multivariate stable distribution simulations. See also \cite{CMS76,MN94,Nol08} for sampling stable distributed random variables. In contrast, the Pareto distribution has a simple probability density and thus can be easily sampled by the classical acceptance and rejection method. Since the stable and the Pareto distribution have the same tail behaviour, and inspired by the stable central limit theorem (see, e.g.\ \cite{Hal81,CNXY19}), we replace the stable noise in \eqref{EM2} with a Pareto distributed noise, and consider the following EM scheme:

Let $\widetilde{Z}_{1},\widetilde{Z}_{2},\cdots$ be an iid sequence of $d$-dimensional random vectors, which are Pareto distributed, i.e.\
\begin{align}\label{density}
    \widetilde Z_1 \sim p(z) = \frac{\alpha}{\sigma_{d-1}|z|^{\alpha+d}}\,
    \II_{(1,\infty)}(|z|);
\end{align}
we denote by $\sigma_{d-1} = {2\pi^{\frac{d}{2}}} / {\Gamma(\frac{d}{2})}$ the surface area of the unit sphere $\mathds{S}^{d-1}\subset\rd$. We will approximate the SDE \eqref{SDE} by the following approximation scheme:
\begin{align}\label{EM}
    \tilde{Y}_0 = x,\quad \tilde{Y}_{k+1} = \tilde{Y}_{k} + \eta b(\tilde{Y}_{k}) + \frac{\eta^{1/\alpha}}{\sigma}\,\widetilde{Z}_{k+1}, \quad k=0,1,2,\dots,
\end{align}
where $\eta>0$ is the step size, $\sigma^\alpha = \alpha /(\sigma_{d-1} C_{d,\alpha})$, and
\begin{gather}\label{normalisation}
    C_{d,\alpha}
    = |\xi|^\alpha\left(\int_{\rd\setminus\{0\}}\left(1-\cos\langle\xi,y\rangle\right) \,\frac{\dif y}{|y|^{\alpha+d}}\right)^{-1}
    = \alpha 2^{\alpha-1} \pi^{-d/2} \frac{\Gamma\left(\frac{d+\alpha}2\right)}{\Gamma\left(1-\frac\alpha 2\right)},
\end{gather}
see e.g. \cite[Example 2.4.d)]{BSW} and \cite[III.18.23]{BF75}. It is easy to see that $(\tilde{Y}_k)_{k \ge 0}$ is a Markov chain.

We aim to study the error bounds in the Wasserstein-$1$ distance for the above two schemes, in particular for large time.

\subsection{Motivation, contribution and method}
The EM approximation of SDEs is a {classical} research topic, both in probability theory and in numerical analysis, and over the past decades there have been many contributions, see for instance \cite{BT96,FG16, Lem07,Sha18,DG20,PP20} for SDEs driven by a Brownian motion, and \cite{JMW96,PT97,Jac04,PT17,MX18,KS19} for SDEs driven by L\'evy noise. Most of these papers focus on error bounds of the solution to the SDE and the EM approximation in a time interval $[0,T]$ for some finite $T>0$; typically, there appears a constant $C_T$ (depending on $T$) in the error bounds, which tends to $\infty$ as $T \to \infty$.

The recent use of Langevin samplings in machine learning, has caused a surge of interest error bounds for the invariant measures of the solution of the SDE and of the EM discretization, see e.g.\ \cite{NSR19,SSG19,ZFM20,CLX21,JPXX21}.
We refer the reader to \cite{LMYY18, SZ21, Tal90} for discrete schemes for the invariant measure of SDEs driven by a Brownian motion.
Panloup \cite{Pan08} uses certain recursive procedures to compute the invariant measure of L\'evy-driven SDEs, but he does not determine the convergence rate. To the best of our knowledge, our paper is the first contribution studying the bound between the invariant measures of solutions to SDEs driven by stable noise and their EM discretizations.

A further motivation of our research is to show that the EM scheme with Pareto distributed innovations can indeed be used to approximate the invariant measures of SDEs driven by an $\alpha$-stable noise with $\alpha \in (1,2)$. In order to speed up the EM scheme, actual implementations of the discretization \eqref{EM}, use iid random variables $(\widetilde Z_k)_{k \ge 1}$ with Pareto distribution rather than stable innovations. The advantage of this approach is that the Pareto distribution has an explicitly given density (see \eqref{density}) which allows for a much simpler sampling than stable random variables. We also show that the convergence rate $\eta^{2/\alpha-1}$ is optimal for the Ornstein--Uhlenbeck process on $\real$.

For $\alpha=2$, the stable process $Z_t$ is (essentially) a $d$-dimensional standard Brownian motion and the convergence rate for the corresponding invariant measure is $\sqrt{\eta}$ (up to a logarithmic correction), see for instance \cite{FSX19}. Our optimal rate $\eta^{\frac 2{\alpha}-1}$ will tend to $O(1)$ rather than $\sqrt{\eta}$ as $\alpha \uparrow 2$; this type of ``phase transition'' has been observed in many situations, e.g.\  in the stable law CLT \cite{Hal81,Xu19}. This is due to the fact that $\alpha$-stable distributions with $\alpha \in (0,2)$ do not have second moments, while the $2$-stable distribution is the Gaussian law having arbitrary moments.

Our approach in proving the main results is via a discrete version of Duhamel principle and Bismut's formula in Malliavin calculus. More precisely, we split the stochastic process $(X_t)_{t \ge 0}$ into smaller pieces $(X_t)_{(k-1) \eta \le t \le k \eta}$ for $k \ge 1$ and replace $(X_t)_{(k-1) \eta \le t \le k \eta}$ with $\tilde{Y}_{k}$ and $Y_k$, respectively.  This procedure is reminiscent to Lindeberg's method for the CLT. In order to bound the error caused by these replacements, we use the semigroup $P_t$ given by $(X_t)_{t \ge 0}$ and study its regularity using Malliavin's calculus for jump processes. In order to bound the second-order derivative of $P_t$, we need to adopt the framework of the time-change argument established in \cite{Z} and use the Bismut formula.

\subsection{Notation}
Whenever we want to emphasize the starting point $X_0=x$ for a given $x \in \rd$, we will write $X_t^x$ instead of $X_t$; we use this also for $Y_k^y$ and $\tilde Y^y_k$ for a given $y \in \rd$. By $P_t,$ $Q_k$ and $\tilde{Q}_k$ we denote the Markov semigroups of $X_t$, $Y_k$ and $\tilde{Y}_k$, respectively, i.e.\
\begin{gather*}
    P_tf(x) = \Ee f(X_t^x),
    \quad
    Q_kf(x) = \Ee f(Y_k^{x}),
    \quad\text{and}\quad
    \tilde{Q}_kf(x) = \Ee f(\tilde{Y}_k^{x}).
\end{gather*}
for a bounded measurable function $f:\rd\to\real$, $x\in\rd$, $t\geq 0$ and $k=0,1,2,\dots$.

As usual, $\mathcal{C}(\rd,\real)$ denote the continuous functions $f: \rd \to \real$, and $\mathcal{C}^2(\rd,\real)$
[$\mathcal{C}_b^2(\rd,\real)$] are the twice continuously differentiable functions [which are bounded together with all their derivatives];  $\nabla f(x)\in \rd$ and $\nabla^2 f(x)\in \real^{d \times d}$ are the gradient and the Hessian. For $v, v_1, v_2,x \in \rd$, the directional derivatives are given by
\begin{align*}
    \nabla_v f(x)
    &= \left\langle \nabla f(x), v\right\rangle
    = \lim_{\varepsilon \to 0} \frac{f(x+\varepsilon v)-f(x)}{\varepsilon},\\
    \nabla_{v_2} \nabla_{v_1} f(x)
    &= \left\langle \nabla^2 f(x), v_1 v^\top_2\right\rangle_{\mathrm{HS}}
    = \lim_{\varepsilon \to 0} \frac{\nabla_{v_1}f(x+\varepsilon v_2)-\nabla_{v_1}f(x)}{\varepsilon},
\end{align*}
where $\left\langle A, B\right\rangle_{\mathrm{HS}}:=\sum_{i,j=1}^d A_{ij} B_{ij}$ for $A, B \in \real^{d \times d}$. The Hilbert-Schmidt norm of a matrix $A\in \real^{d \times d}$ is $\|A\|_{\mathrm{HS}}=\sqrt{\sum_{i,j=1}^dA_{ij}^2}$.

The directional derivatives are similarly defined for (sufficiently smooth) vector-valued functions $f=(f_{1},f_{2},\cdots,f_{d})^{\top}: \rd\to \rd$:
let $v, v_1, v_2,x \in \rd$, then
$\nabla_{v}f(x)=(\nabla_{v}f_{1},\nabla_{v},\dots,\nabla_{v}f_{d})^{\top}$,
$\nabla_{v_{2}}\nabla_{v_{1}}f(x)= (\nabla_{v_{2}}\nabla_{v_{1}}f_{1},\dots,\nabla_{v_{2}} \nabla_{v_{1}}f_{d})^{\top}$.

For $f \in \mathcal{C}_b^2(\rd,\real)$, we will use the supremum and the supremum Hilbert-Schmidt norm
\begin{gather*}
    \|\nabla f\|_{\infty} =
    \sup_{x\in\rd}|\nabla f(x)|,
\quad
    \|\nabla^2f\|_{\mathrm{HS},\infty} =
    \sup_{x\in\rd}\|\nabla^2f(x)\|_{\mathrm{HS}}.
\end{gather*}

The Wasserstein-$1$ distance between two probability measures $\mu_{1}$ and $\mu_{2}$ on $\rd$ is defined as
\begin{align}\label{deW1}
    W_{1}(\mu_{1},\mu_{2})=\inf_{(X,Y)\in\mathsf{C}(\mu_{1},\mu_{2})}\Ee |X-Y|,
\end{align}
where $\mathsf{C}(\mu_{1},\mu_{2})$ is the set of all coupling realizations of $\mu_1, \mu_2$, i.e.\ all random variables with values in $\real^{2d}$ with marginals $\mu_{1},\mu_{2}$. We also have the following dual description of the Wasserstein distance
\begin{gather*}
    W_{1}(\mu_{1},\mu_{2})=\sup_{h \in \mathrm{Lip}(1)} |\mu_{1}(h)-\mu_{2}(h)|,
\end{gather*}
where $\mathrm{Lip}(1)=\{h: \rd \to \real; \ |h(y)-h(x)| \le |y-x|\}$ and $\mu_{i}(h)=\int_{\real} h(x)\,\mu_{i}(\dif x)$, $i=1,2$.

We will frequently need the following weight function
\begin{gather*}
    V_{\beta}(x)=(1+|x|^{2})^{\beta/2}, \quad x \in \rd, \; \beta\geq 0.
\end{gather*}

Finally, we write $\lfloor x\rfloor$ for the largest integer which is less than or equal to $x\in\real$, and throughout $C_{d,\alpha}$ is the constant \eqref{normalisation}.

\subsection{Assumptions and main results}

Throughout this paper, we make the following assumption:
\begin{assA*}
The function $b:\rd\to\rd$ is twice continuously differentiable and there exist constants $\theta_{1},\theta_{2} > 0$ and $\theta_{3},K \geq 0$ such that
\begin{align}\label{diss}
    \left\langle b(x)-b(y),x-y\right\rangle\leq-\theta_{1}|x-y|^{2}+K \quad \forall x,y\in\rd
\intertext{and}\label{gra}
    |\nabla_{v} b(x)| \leq \theta_{2}\left|v\right|, \quad
    |\nabla_{v_{1}}\nabla_{v_{2}}b(x)| \leq \theta_{3}\left|v_{1}\right|\left|v_{2}\right| \quad \forall v,v_{1},v_{2},x\in\rd.
\end{align}
\end{assA*}

\begin{remark}
Note that \eqref{gra} immediately implies the following linear growth condition
\begin{align}\label{linear}
    |b(x)-b(0)|\leq\theta_{2}|x|,\quad x\in\rd.
\end{align}
\end{remark}

Under \textbf{Assumption A}, we will show that both $(X_{t})_{t \ge 0}$ and $(\tilde{Y}_{k})_{k \ge 0}$ are ergodic; we write $\mu$ and $\tilde{\mu}_{\eta}$, respectively, for their invariant measures, see Propositions~\ref{W1.1-0} and~\ref{W1.1-1} below. Throughout the paper the constants $C$, $c_{1}$, $c_{2}$, $c_{3}$, $c_{4}$ and $\lambda$ may depend on $\theta_{1},\theta_{2},\theta_{3},K,\alpha,d,|b(0)|$ {and $\beta$ for some constant $\beta\in[1,\alpha)$}, but we often suppress this in our notation; moreover, the  exact values of the constants may vary from line to line. Our \emph{main results} are the following two theorems:
\begin{theorem}\label{main}
    Let $(X_{t})_{t \ge 0}$ and $(\tilde{Y}_{k})_{k \ge 0}$ be defined by \eqref{SDE} and \eqref{EM} \textup{(}step size $\eta$\textup{)}, and denote by $\mu$ and $\tilde{\mu}_{\eta}$, their invariant measures. Under \textbf{Assumption A}, there exists a constant $C$ such that the following two statements hold:
    \begin{enumerate}
    \item\label{main-i}
        For every $N\geq 2$ and step size $\eta<\min\left\{1,\:\theta_{1}/{(8\theta_{2}^{2})},\: {1}/{\theta_{1}}\right\}$, one has
        \begin{gather*}
            W_{1}\bigl(\mathrm{law}(X_{\eta N}),\mathrm{law}(\tilde{Y}_{N})\bigr) \leq C(1+|x|) \eta^{2/\alpha-1}.
        \end{gather*}

    \item\label{main-ii}
        For every step size $\eta < \min\left\{1,\:{\theta_{1}}/{\theta_{2}^{2}},\:{1}/{\theta_{1}}\right\}$, one has
        \begin{align*}
            W_{1}\bigl(\mu,\tilde{\mu}_{\eta}\bigr) \leq C\eta^{2/\alpha-1}.
        \end{align*}
    \end{enumerate}
\end{theorem}
\begin{theorem}\label{main2}
    Let $(X_{t})_{t \ge 0}$ and $(Y_{k})_{k \ge 0}$ be defined by \eqref{SDE} and \eqref{EM2} \textup{(}step size $\eta$\textup{)}, and denote by $\mu$ and $\mu_{\eta}$, their invariant measures. Under \textbf{Assumption A}, for any
    $\beta\in[1,\alpha)$, there exists a constant $C$ depending on $\beta$
    such that the following two statements hold:
    \begin{enumerate}
    \item\label{main2-i}
        For every $N\geq 2$ and step size $\eta<\min\left\{1,\:\theta_{1}/{(8\theta_{2}^{2})},\: {1}/{\theta_{1}}\right\}$, one has
        \begin{gather*}
            W_{1}\bigl(\mathrm{law}(X_{\eta N}),\mathrm{law}(Y_{N})\bigr) \leq C(1+|x|^{\beta}) \eta^{1+\frac{1}{\alpha}-\frac{1}{\beta}}.
        \end{gather*}

    \item\label{main2-ii}
        For every step size $\eta < \min\left\{1,\:{\theta_{1}}/{\theta_{2}^{2}},\:{1}/{\theta_{1}}\right\}$, one has
        \begin{align*}
            W_{1}\bigl(\mu,\mu_{\eta}\bigr) \leq C \eta^{1+\frac{1}{\alpha}-\frac{1}{\beta}}.
        \end{align*}
    \end{enumerate}
\end{theorem}

\begin{remark}
    The rate $\eta^{2/\alpha-1}$ in the first theorem is optimal for the one-dimensional Ornstein--Uhlenbeck process,
    see Proposition~\ref{ourate} below.
\end{remark}

The proofs of Theorems \ref{main} and \ref{main2} are presented in Section \ref{mainproof}. In Section \ref{A}, we use a time-change argument and the Bismut formula to prove Lemma \ref{regular}, which is the key to the proof of our main result.  Appendix~\ref{sec2} includes the proofs of the propositions in this section for the completeness. Finally, in Appendix \ref{OU}, the exact convergence rate $\eta^{2/\alpha-1}$ is reached for the Ornstein--Uhlenbeck process on $\real$, which shows that the rate in Theorem  \ref{main}\,\eqref{main-ii} is sharp.

\subsection{Auxiliary propositions}
Here we collect a few auxiliary properties of $(X_{t})_{t \ge 0}$ and $(Y_{k})_{k \ge 0}$. The proofs are standard, but we include them in Appendix \ref{sec2} to be self-contained. Recall that $V_{\beta}(x)=(1+|x|^{2})^{\beta/2}$.
\begin{proposition}\label{W1.1-0}
    Let \textbf{Assumption A} hold and denote by $(X_{t})_{t\geq 0}$ the solution to the SDE \eqref{SDE}. Then, $(X_{t})_{t \ge 0}$ admits a unique invariant probability measure $\mu$ such that for $1\leq \beta < \alpha$
    \begin{align}\label{exact1}
        \sup_{|f|\leq V_{\beta}}\big| \Ee [f(X_{t}^x)]-\mu(f)\big|
        \leq c_{1} V_{\beta}(x) \eup^{-c_{2}t}, \quad t>0,
    \end{align}
    for some constants $c_{1},c_{2}>0$. In particular, there exists a constant $C>0$ such that
    \begin{align}\label{moment}
        \Ee |X_{t}^x|^{\beta}\leq C(1+|x|^{\beta}), \quad t>0.
    \end{align}
\end{proposition}

\begin{proposition} \label{W1.1}
    Under \textbf{Assumption A}, there exist for every $t>0$ and $x,y\in\rd$ constants $C>0$ and $\lambda>0$ such that
    \begin{align*}
        W_{1}\left(\mathrm{law}(X_{t}^{x}),\mathrm{law}(X_{t}^{y})\right)
        \leq C\eup^{-\lambda t}|x-y|.
    \end{align*}
\end{proposition}
\begin{proposition}\label{W1.1-1}
    Let \textbf{Assumption A} hold and denote by $(Y_{k})_{k \ge 0}$ and $(\tilde{Y}_{k})_{k \ge 0}$ the Markov chains defined by \eqref{EM2} and \eqref{EM}, respectively. Assume that the step size satisfies   $\eta < \min\left\{1,\: \theta_{1}/\theta_{2}^{2},\:1/\theta_{1}\right\}$. Then
     \begin{enumerate}
    \item  the chain $(Y_{k})_{k \ge 0}$ admits a unique invariant measure $\mu_{\eta}$, such that for all $x\in\rd$ and $k>0$,
    \begin{align}\label{exact22}
        \sup_{|f|\leq V_{1}}|\Ee f(Y_{k}^{x})-\mu_{\eta}(f)|
        \leq c_{1}V_{1}(x)\eup^{-c_{2}k},
    \end{align}
    for some constants $c_{1},c_{2}>0$.
    \item
  the chain $(\tilde{Y}_{k})_{k \ge 0}$ admits a unique invariant measure $\tilde{\mu}_{\eta}$, such that for all $x\in\rd$ and $k>0$,
    \begin{align}\label{exact2}
        \sup_{|f|\leq V_{1}}|\Ee f(\tilde{Y}_{k}^{x})-\tilde{\mu}_{\eta}(f)|
        \leq c_{3}V_{1}(x)\eup^{-c_{4}k},
    \end{align}
    for some constants $c_{3},c_{4}>0$.
    \end{enumerate}
\end{proposition}


\begin{lemma}\label{E-Mmoment}
   Let \textbf{Assumption A} hold and denote by $(Y_{k})_{k \ge 0}$ and $(\tilde{Y}_{k})_{k \ge 0}$ the Markov chains defined by \eqref{EM2} and \eqref{EM}, respectively. If the step size satisfies   $\eta < \min\left\{1,\frac{\theta_{1}}{8\theta_{2}^{2}},\frac{1}{\theta_{1}}\right\}$,
   then there is a constant $C>0$, which is independent of $\eta$, such that
    \begin{align}\label{momentd2}
        \Ee |Y_{k}^{x}|^{\beta} &\leq C(1+|x|^{\beta}),\\
    \label{momentd}
        \Ee |\tilde{Y}_{k}^{x}| &\leq C(1+|x|),
    \end{align}
   hold  for any $\beta\in[1,\alpha)$, $x\in\mathds{R}^{d}$ and $k>0$.
\end{lemma}

\section{Proof of Theorems \ref{main} and \ref{main2}}\label{mainproof}

We begin with several auxiliary lemmas which will be used to prove Theorems \ref{main} and \ref{main2}.

\subsection{Auxiliary lemmas}

The first auxiliary lemma is about the regularity of the semigroup induced by $(X_t)_{t \ge 0}$.
\begin{lemma}\label{regular}
    Let $h\in\mathrm{Lip}(1)$ and $X_{t}^{x}$ be the solution to the SDE \eqref{SDE}. For all vectors $v,v_{1},v_{2}\in\rd$ and $t\in(0,1]$, we have
    \begin{gather}\label{semi1}
        |\nabla_{v}P_{t}h(x)|
        \leq \eup^{\theta_{2}}\left|v\right|
    \intertext{and}\label{semi2}
        |\nabla_{v_{2}}\nabla_{v_{1}}P_{t}h(x)|
        \leq C t^{-1/\alpha}\left|v_{1}\right|\left|v_{2}\right|,
    \end{gather}
    for some constant $C>0$.
\end{lemma}

{
\begin{remark}
We will prove the above lemma using Malliavin's calculus \cite{Nor86} in Section \ref{A} further down. A very careful and knowledgeable referee suggested a shorter argument which avoids the use of Bismut's formula and Norris' formalism. The idea of
her/his proof is to use the so called "variance of constants method" to establish the
corresponding Bismut--Elworthy--Li formula. The alternative proof of Lemma \ref{regular} will be added
in Appendix \ref{supple} below.
\end{remark}
}

Using the inequalities \eqref{moment} and \eqref{momentd2}, we can obtain the following estimates:
\begin{lemma}
Let $(X_{t})_{t\geq 0}$ be the solution to the SDE \eqref{SDE} and $(Y_{k})_{k \ge 0}$ be the Markov chains defined by \eqref{EM2}. If the step size satisfies   $\eta < \min\left\{1,\frac{\theta_{1}}{8\theta_{2}^{2}},\frac{1}{\theta_{1}}\right\}$, then the following estimates
hold for all $t\in(0,1]$, $\beta\in[1,\alpha)$:
\begin{align}\label{1}
    \Ee |Y_{1}^{x}-x|^{\beta} &\leq C(1+|x|^{\beta})\eta^{\beta/\alpha},\\
\label{2}
    \Ee |X_{t}^{x}-x|^{\beta} &\leq C(1+|x|^{\beta})t^{\beta/\alpha},\\
\label{3}
    \Ee |X_{\eta}^{x}-Y_{1}^{x}|^{\beta} &\leq C(1+|x|^{\beta})\eta^{\beta+\frac{\beta}{\alpha}}.
\end{align}
\end{lemma}

\begin{proof}
The first inequality follows immediately from
\begin{align*}
    \Ee |Y_{1}^{x}-x|^{\beta}
    = \Ee\left|\eta b(x) + Z_{\eta}\right|^{\beta}
    \leq 2\left[\eta^{\beta}|b(x)|^{\beta}+\Ee|Z_{\eta}|^{\beta}\right]
    \leq C(1+|x|^{\beta})\eta^{\beta/\alpha}.
\end{align*}

From the H\"{o}lder inequality and \eqref{moment}, we obtain
\begin{align*}
    \Ee|X_{t}^{x}-x|^{\beta}
    &=\Ee\left|\int_{0}^{t}b(X_{s}^{x})\,\dif s+Z_{t}\right|^{\beta}\\
    &\leq 2\Ee\left|\int_{0}^{t}b(X_{s}^{x})\,\dif s\right|^{\beta}+2|Z_{t}|^{\beta}\\
    &\leq 2t^{\beta-1}\int_{0}^{t}\Ee|b(X_{s}^{x})|^{\beta}\,\dif s+2\Ee|Z_{1}|^{\beta}t^{\frac{\beta}{\alpha}}\\
    &\leq C(1+|x|^{\beta})t^{\beta/\alpha},
\end{align*}
which implies the second inequality.

For the last inequality, the H\"{o}lder inequality, \eqref{linear} and \eqref{2} imply
\begin{align*}
    \Ee |X_{\eta}^{x}-Y_{1}^{x}|^{\beta}
    &=\Ee \left|\int_{0}^{\eta}[b(X_{s}^{x})-b(x)]\,\dif s\right|^{\beta}\\
    &\leq \eta^{\beta-1}\int_{0}^{\eta}\Ee|b(X_{s}^{x})-b(x)|^{\beta}\,\dif s\\
    &\leq \theta\eta^{\beta-1}\int_{0}^{\eta}\Ee|X_{s}^{x}-x|^{\beta}\,\dif s\\
    &\leq C(1+|x|^{\beta})\eta^{\beta-1}\int_{0}^{\eta}s^{\frac{\beta}{\alpha}}\,\dif s\\
    &\leq C(1+|x|^{\beta})\eta^{\beta+\frac{\beta}{\alpha}}.
    \qedhere
\end{align*}
\end{proof}

In order to prove Theorem \ref{main}, we need the following two lemmas. The first is just an intermediate step for the proof of the second lemma, which is the key to proving Theorem~\ref{main}. Notice that the fractional Laplacian operator $(-\Delta)^{\alpha/2}$ is the infinitesimal generator  of the rotationally invariant $\alpha$-stable L\'evy process process $(Z_{t})_{t\geq0}$, which is defined as a principal value (p.v.) integral: for any $f\in \mathcal{C}^2(\rd,\real)$,
\begin{align}\label{frac1}
    (-\Delta)^{\alpha/2}f(x)
    = C_{d,\alpha}\cdot\mathrm{p.v.}\int_{\rd} \left(f(x+y)-f(x)\right) \frac{\dif y}{|y|^{\alpha+d}}.
\end{align}
\begin{lemma}\label{laplace}
    Let $\alpha\in(1,2)$ and $f:\mathds{R}^{d}\to \mathds{R}$ satisfying $\|\nabla f\|_{\infty}<\infty$ and $\|\nabla^{2}f\|_{\mathrm{HS},\infty}<\infty$. For all $x,y\in\rd$ one has
    \begin{equation}\label{crucialbis}
        \left|(-\Delta)^{\alpha/2}f(x) - (-\Delta)^{\alpha/2}f(y)\right|
        \leq
        \frac{C_{d,\alpha}\|\nabla^{2}f\|_{\mathrm{HS},\infty}\sigma_{d-1}}{(2-\alpha)(\alpha-1)}|x-y|^{2-\alpha}.
    \end{equation}
\end{lemma}

\begin{proof}
From the definition of the fractional Laplacian \eqref{frac1} and the symmetry of the representing measure we have for any $R>0$
\begin{align*}
    (-\Delta)^{\alpha/2}f(x)
    &= C_{d,\alpha}\int_{\mathds{S}^{d-1}}\int_{0}^{\infty}\frac{f(x+r\theta)-f(x)-r\left\langle\theta,\nabla f(x)\right\rangle\II_{(0,R)}(r)}{r^{\alpha+1}}\,\dif r\,\dif\theta\\
    &= C_{d,\alpha}\int_{\mathds{S}^{d-1}} \int_{0}^{R}\int_{0}^{r}\frac{\left\langle\theta,\nabla f(x+\theta s)-\nabla f(x)\right\rangle}{r^{\alpha+1}}\,\dif s\,\dif r\,\dif \theta\\
    &\quad\mbox{}+C_{d,\alpha}\int_{\mathds{S}^{d-1}} \int_{R}^{\infty}\int_{0}^{r}\frac{\left\langle\theta,\nabla f(x+\theta s)\right\rangle}{r^{\alpha+1}}\,\dif s\,\dif r\,\dif \theta
\end{align*}
Then, for all $x,y\in\rd$,
\begin{align*}
    &\big|(-\Delta)^{\alpha/2}f(x) - (-\Delta)^{\alpha/2}f(y)\big|\\
    &\quad\leq C_{d,\alpha}\int_{\mathds{S}^{d-1}} \int_{0}^{R}\int_{0}^{r}\frac{\left|\nabla f(x+\theta s)-\nabla f(x)-\nabla f(y+\theta s)+\nabla f(y)\right|}{r^{\alpha+1}}\,\dif s\,\dif r\,\dif \theta\\
    &\qquad\mbox{}+C_{d,\alpha}\int_{\mathds{S}^{d-1}} \int_{R}^{\infty}\int_{0}^{r}\frac{\left|\nabla f(x+\theta s)-\nabla f(y+\theta s)\right|}{r^{\alpha+1}}\,\dif s\,\dif r\,\dif \theta.
\end{align*}
For the first integral we have
\begin{align*}
    &C_{d,\alpha}\int_{\mathds{S}^{d-1}} \int_{0}^{R}\int_{0}^{r}\frac{\left|\nabla f(x+\theta s)-\nabla f(x)-\nabla f(y+\theta s)+\nabla f(y)\right|}{r^{\alpha+1}}\,\dif s\,\dif r\,\dif \theta\\
    &\quad\leq C_{d,\alpha}\int_{\mathds{S}^{d-1}} \int_{0}^{R}\int_{0}^{r}\frac{\left|\nabla f(x+\theta s)-\nabla f(x)\right|+\left|\nabla f(y+\theta s)-\nabla f(y)\right|}{r^{\alpha+1}}\,\dif s\,\dif r\,\dif \theta\\
    &\quad\leq 2C_{d,\alpha}\|\nabla^{2}f\|_{\mathrm{HS},\infty}\int_{\mathds{S}^{d-1}}\int_{0}^{r}
    \frac{s}{r^{\alpha+1}}\,\dif s\,\dif r\,\dif\theta=\frac{C_{d,\alpha}\|\nabla^{2}f\|_{\mathrm{HS},\infty}\sigma_{d-1}}{2-\alpha}\,R^{2-\alpha},
\end{align*}
and for the second term we get
\begin{align*}
    &C_{d,\alpha}\int_{\mathds{S}^{d-1}} \int_{R}^{\infty}\int_{0}^{r}\frac{\left|\nabla f(x+\theta s)-\nabla f(y+\theta s)\right|}{r^{\alpha+1}}\,\dif s\,\dif r\,\dif \theta\\
    &\quad\leq C_{d,\alpha}\|\nabla^{2}f\|_{\mathrm{HS},\infty}\int_{\mathds{S}^{d-1}}
    \int_{R}^{\infty}\int_{0}^{r}\frac{|x-y|}
    {r^{\alpha+1}}\,\dif s\,\dif r\,\dif\theta\\
    &\quad= \frac{C_{d,\alpha}\|\nabla^{2}f\|_{\mathrm{HS},\infty}\sigma_{d-1}}{\alpha-1}\,|x-y|R^{1-\alpha}.
\end{align*}
Hence, the assertion follows upon taking $R=|x-y|$.
\end{proof}

\begin{lemma}\label{compare}
    Let $(X_{t})_{t\geq 0}$ and $(\tilde{Y}_{k})_{k\geq 0}$ be defined by \eqref{SDE} and \eqref{EM}, respectively. There exists a constant $C>0$ such that for all $x\in\rd$, $\eta\in(0,1)$, $f:\mathds{R}^{d}\to \mathds{R}$ satisfying $\|\nabla f\|_{\infty}<\infty$ and $\|\nabla^{2}f\|_{\mathrm{HS},\infty}<\infty$,
    \begin{align*}
        |P_\eta f(x)-\tilde{Q}_1f(x)|
        \leq C(1+|x|)\left(\|\nabla f\|_{\infty}+\|\nabla^{2}f\|_{\mathrm{HS},\infty}\right)\eta^{2/\alpha}.
    \end{align*}
\end{lemma}

\begin{proof}
From \eqref{SDE} and \eqref{EM}, we see
\begin{align*}
    \Ee [f(X_{\eta}^{x})-f(\tilde{Y}_{1})]
    &=\Ee \left[f\left(x+\int_{0}^{\eta}b(X_{r}^{x})\,\dif r +Z_{\eta}\right)
        - f\left(x+\eta b(x) +\frac{\eta^{1/\alpha}}{\sigma} \widetilde{Z}\right)\right]\\
    &=\mathsf{J}_{1}+\mathsf{J}_{2},
\end{align*}
where
\begin{gather*}
    \mathsf{J}_{1}
    :=\Ee \left[f\left(x+\int_{0}^{\eta} b(X_{r}^{x})\,\dif r+Z_{\eta}\right)
        -f\left(x+\eta b(x)+Z_{\eta}\right)\right],
    \\
    \mathsf{J}_{2}
    := \Ee \left[f\left(x+\eta b(x)+Z_{\eta}\right)-f\left(x+\eta b(x)\right)\right]
        -\Ee \left[f\left(x+\eta b(x)+\frac{\eta^{1/\alpha}}{\sigma} \widetilde{Z}\right)-f\left(x+\eta b(x)\right)\right].
\end{gather*}
We can bound $\mathsf{J}_{1}$ using \eqref{gra} and \eqref{2} with $\beta=1$:
\begin{align*}
    |\mathsf{J}_{1}|
    &\leq\|\nabla f\|_{\infty}\Ee \left|\int_{0}^{\eta}b(X_{r}^{x})\,\dif r-\eta b(x)\right|\\
    &\leq\|\nabla f\|_{\infty}\int_{0}^{\eta}\Ee |b(X_{r}^{x})-b(x)|\,\dif r\\
    &\leq\theta_{2}\|\nabla f\|_{\infty}\int_{0}^{\eta}\Ee |X_{r}^{x}-x|\,\dif r\\
    &\leq C\theta_{2}(1+|x|)\|\nabla f\|_{\infty}\int_{0}^{\eta}r^{1/\alpha}\,\dif r\\
    &\leq C(1+|x|)\| \, \nabla f\|_{\infty}\eta^{1+1/\alpha}.
\end{align*}
For the first term of $\mathsf{J}_{2}$ we use Dynkin's formula (see e.g.\ \cite{CX19}) to get
\begin{align*}
    \Ee \big[f\big(x+\eta b(x)+Z_{\eta}\big)-f\big(x+\eta b(x)\big)\big]
    = \int_{0}^{\eta}\Ee \big[(-\Delta)^{\alpha/2}f\big(x+\eta b(x)+Z_{r}\big)\big]\,\dif r.
\end{align*}
For the second part of $\mathsf{J}_{2}$ we use that $C_{d,\alpha}= \alpha \sigma_{d-1}^{-1} \sigma^{-\alpha}$ and Taylor's formula to see
\begin{align*}
    &\Ee \left[f\left(x+\eta b(x)+\frac{\eta^{1/\alpha}}{\sigma}\widetilde{Z}\right)-f\big(x+\eta b(x)\big)\right]\\
    &\quad = \frac{\eta^{1/\alpha}}{\sigma}\Ee \left[\int_{0}^{1}\left\langle\nabla f\left(x+\eta b(x)+ \frac{\eta^{1/\alpha}}{\sigma}t\widetilde{Z}\right),\: \widetilde{Z}\right\rangle \dif t\right]\\
    &\quad = \frac{\eta^{1/\alpha}}{\sigma} \int_{|z|\geq1} \int_{0}^{1} \alpha \left\langle\nabla f\big(x+\eta b(x)+\frac{\eta^{1/\alpha}} {\sigma}tz\big),\: z\right\rangle \frac{\dif t\,\dif z}{\sigma_{d-1} |z|^{\alpha+d}}\\
    &\quad = \frac{\alpha\eta}{\sigma_{d-1} \sigma^{\alpha}}\int_{|z|\geq\sigma^{-1}\eta^{1/\alpha}}\int_{0}^{1}
    \left\langle\nabla f\big(x+\eta b(x)+tz\big),\:z\right\rangle  \frac{\dif t\,\dif z}{|z|^{\alpha+d}}\\
    &\quad = \eta(-\Delta)^{\alpha/2}f(x+\eta b(x))-\mathsf{R},
\end{align*}
where
\begin{align*}
    \mathsf{R}
    := \eta C_{d,\alpha}\int_{|z|<\sigma^{-1}\eta^{1/\alpha}}\int_{0}^{1}\left\langle\nabla f\big(x+\eta b(x)+tz\big),\: z\right\rangle
    \frac{\dif t\,\dif z}{|z|^{\alpha+d}}.
\end{align*}
Together, the above estimates yield
\begin{align*}
    |\mathsf{J}_{2}|
    \leq |\mathsf{R}| + \left|\int_{0}^{\eta} \Ee \big[(-\Delta)^{\alpha/2} f\big(x+\eta b(x)+Z_{r}\big)\big]\,\dif r-\eta(-\Delta)^{\alpha/2}f(x+\eta b(x))\right|.
\end{align*}
Further, we have
\begin{align*}
    |\mathsf{R}|
    &= \eta C_{d,\alpha} \left|\int_{|z|<\sigma^{-1}\eta^{1/\alpha}}\int_{0}^{1} \left\langle\nabla f\big(x+\eta b(x)+tz\big)-\nabla f\big(x+\eta b(x)\big),\: z\right\rangle \frac{\dif t\,\dif z}{|z|^{\alpha+d}} \right|\\
    &\leq \eta C_{d,\alpha}\int_{|z|<\sigma^{-1}\eta^{1/\alpha}}\int_{0}^{1} \left|\nabla f\big(x+\eta b(x)+tz\big)-\nabla f\big(x+\eta b(x)\big)\right| \frac{\dif t\,\dif z}{|z|^{\alpha+d-1}}\\
    &\leq \frac12\,\eta C_{d,\alpha} \|\nabla^{2}f\|_{\mathrm{HS},\infty} \int_{|z|<\sigma^{-1}\eta^{1/\alpha}} \frac{\dif z}{|z|^{\alpha+d-2}}
    \leq C\|\nabla^{2}f\|_{\mathrm{HS},\infty}\,\eta^{2/\alpha}.
\end{align*}
By Lemma \ref{laplace}, we also have
\begin{align*}
    &\left|\int_{0}^{\eta}\Ee \left[(-\Delta)^{\alpha/2}f\big(x+\eta b(x)+Z_{r}\big)\right] \dif r
    - \eta(-\Delta)^{\alpha/2}f(x+\eta b(x))\right|\\
    &\quad \leq\int_{0}^{\eta}\Ee \left|(-\Delta)^{\alpha/2}f\big(x+\eta b(x)+Z_{r}\big)\big] -(-\Delta)^{\alpha/2}f(x+\eta b(x))\right| \dif r\\
    &\quad \leq C\|\nabla^{2}f\|_{\mathrm{HS},\infty} \int_{0}^{\eta}\Ee  \left[|Z_{r}|^{2-\alpha}\right] \dif r\\
    &\quad =C\|\nabla^{2}f\|_{\mathrm{HS},\infty}\int_{0}^{\eta}\Ee \left[|Z_{1}|^{2-\alpha}\right] r^{2/\alpha-1}\,\dif r\\
    &\quad\leq C \Ee \left[|Z_{1}|^{2-\alpha}\right] \|\nabla^{2}f\|_{\mathrm{HS},\infty}\,\eta^{2/\alpha}.
\end{align*}
The proof follows if we combine all estimates.
\end{proof}

In order to prove Theorem \ref{main2}, we need two more lemmas. The first is just an intermediate step for the proof of the second lemma, which is the key to proving Theorem~\ref{main2}.
\begin{lemma}\label{holder2}
    Assume that $f$ satisfies $\|\nabla f\|_{\infty}<\infty$ and $\|\nabla^{2}f\|_{\mathrm{HS},\infty}<\infty$. For any $\beta\in[1,2]$ and $x,y\in\rd$, we have
    \begin{align*}
        \left|\nabla f(x)-\nabla f(y)\right|\leq\left(2\|\nabla f\|_{\infty}+\|\nabla^{2}f\|_{\mathrm{HS},\infty}\right)|x-y|^{\beta-1}.
    \end{align*}
\end{lemma}

\begin{proof}
For $|x-y|>1$ we have
\begin{gather*}
    \left|\nabla f(x)-\nabla f(y)\right|\leq 2\|\nabla f\|_{\infty}\leq\|\nabla f\|_{\infty}|x-y|^{\beta-1},
\intertext{and for $|x-y|\leq 1$ we have}
    \left|\nabla f(x)-\nabla f(y)\right|\leq\|\nabla^{2}f\|_{\mathrm{HS},\infty}|x-y|\leq\|\nabla^{2}f\|_{\mathrm{HS},\infty}|x-y|^{\beta-1}.
\qedhere
\end{gather*}
\end{proof}

\begin{lemma}\label{taylor}
    Let $(X_{t})_{t\geq 0}$ and $(Y_{k})_{k\geq 0}$ be defined by \eqref{SDE} and \eqref{EM2}, respectively. There exists a constant $C>0$ such that for all $x\in\rd$, $\eta\in(0,1)$, $\beta\in[1,\alpha)$ and $f:\rd\to \mathds{R}$ satisfying $\|\nabla f\|_{\infty}<\infty$ and $\|\nabla^{2}f\|_{\mathrm{HS},\infty}<\infty$,
    \begin{align*}
        |P_\eta f(x)-Q_1f(x)|
        \leq C (1+|x|^{\beta}) \left(\|\nabla f\|_{\infty}+\|\nabla^{2}f\|_{\mathrm{HS},\infty}\right) \eta^{2+\frac{1}{\alpha}-\frac{1}{\beta}}.
    \end{align*}
\end{lemma}

\begin{proof}
We use a Taylor expansion to get
\begin{align*}
    &\Ee f(X_{\eta}^{x})-\Ee f(Y_{1}^{x})\\
    &=\Ee\left\langle\nabla f(Y_{1}^{x}),X_{\eta}^{x}-Y_{1}^{x}\right\rangle+\Ee\int_{0}^{1}\left\langle\nabla f\left(Y_{1}^{x}+r(X_{\eta}^{x}-Y_{1}^{x})\right)-\nabla f(Y_{1}^{x}),X_{\eta}^{x}-Y_{1}^{x}\right\rangle\dif r\\
    &=\Ee\left\langle\nabla f\left(x+\eta b(x)+Z_{\eta}\right)-\nabla f\left(x+\eta b(x)\right),X_{\eta}^{x}-Y_{1}^{x}\right\rangle+\Ee\left\langle\nabla f\left(x+\eta b(x)\right),X_{\eta}^{x}-Y_{1}^{x}\right\rangle\\
    &\qquad\mbox{}+\Ee\int_{0}^{1}\left\langle\nabla f\left(Y_{1}^{x}+r(X_{\eta}^{x}-Y_{1}^{x})\right)-\nabla f(Y_{1}^{x}),X_{\eta}^{x}-Y_{1}^{x}\right\rangle\dif r\\
    &=:\mathsf{I}+\mathsf{II}+\mathsf{III}.
\end{align*}

For the first term $\mathsf{I}$ we have
\begin{align*}
    \mathsf{I}
    &=\Ee\left\langle\nabla f\left(x+\eta b(x)+Z_{\eta}\right)-\nabla f\left(x+\eta b(x)\right),X_{\eta}^{x} - Y_{1}^{x}\right\rangle\left(\II_{(0,1]}(|Z_{\eta}|)+\II_{(1,\infty)}(|Z_{\eta}|)\right)\\
    &=:\mathsf{I}_{1}+\mathsf{I}_{2}.
\end{align*}
We need the following estimates for the truncated moment of order $\lambda > \alpha$ and the tail of the $\alpha$-stable random variable $Z_\eta$ and $\eta\leq 1$:
\begin{gather*}
    \mathds{P}\left(|Z_\eta| > 1\right) \leq c\eta
    \quad\text{and}\quad
    \Ee\left[|Z_\eta|^\lambda \II_{(0,1]}(Z_\eta)\right] \leq C\eta.
\end{gather*}
Both estimates follow from a straightforward calculation using the standard estimate $q_{\alpha}(\eta,x)\leq C\eta / (\eta^{1/\alpha}+|x|)^{\alpha+d}$ for the density of $Z_\eta$, see e.g.\ \cite[Theorem 2.1]{BG60}. Since $\frac{\beta - 1}{\beta} > \alpha$, we can use the H\"{o}lder inequality and \eqref{3} to get
\begin{align*}
    |\mathsf{I}_{1}|
    &\leq\Ee\left[\left|\nabla f\left(x+\eta b(x)+Z_{\eta}\right)-\nabla f\left(x+\eta b(x)\right)\right| \II_{(0,1]}(|Z_{\eta}|)\left|X_{\eta}^{x}-Y_{1}^{x}\right|\right]\\
    &\leq\|\nabla^{2}f\|_{\mathrm{HS},\infty}\Ee\left[|Z_{\eta}|\II_{(0,1]}(|Z_{\eta}|)\left|X_{\eta}^{x}-Y_{1}^{x}\right|\right]\\
    &\leq\|\nabla^{2}f\|_{\mathrm{HS},\infty} \left(\Ee\left[|Z_{\eta}|^{\frac{\beta}{\beta-1}}\II_{(0,1]}(|Z_{\eta}|)\right]\right)^{\frac{\beta-1}{\beta}}
    \left(\Ee \left|X_{\eta}^{x}-Y_{1}^{x}\right|^{\beta}\right)^{\frac{1}{\beta}}\\
    &\leq C (1+|x|) \|\nabla^{2}f\|_{\mathrm{HS},\infty}
    \eta^{\frac{\beta-1}{\beta}}\, \eta^{1+\frac{1}{\alpha}}\\
    &=C (1+|x|)\|\nabla^{2}f\|_{\mathrm{HS},\infty}\eta^{2+\frac{1}{\alpha}-\frac{1}{\beta}},
\end{align*}
whereas by the H\"{o}lder inequality
\begin{align*}
    |\mathsf{I}_{2}|
    &\leq\Ee\left[\left|\nabla f\left(x+\eta b(x)+Z_{\eta}\right)-\nabla
    f\left(x+\eta b(x)\right)\right|\II_{(1,\infty)}(|Z_{\eta}|) \left|X_{\eta}^{x}-Y_{1}^{x}\right| \right]\\
    &\leq2\|\nabla f\|_{\infty} \Ee\left[\II_{(1,\infty)}(|Z_{\eta}|)\left|X_{\eta}^{x}-Y_{1}^{x}\right|\right]\\
    &\leq2\|\nabla f\|_{\infty} \left(\Ee\II_{(1,\infty)}(|Z_{\eta}|)\right)^{\frac{\beta-1}{\beta}} \left(\Ee\left|X_{\eta}^{x}-Y_{1}^{x}\right|^{\beta}\right)^{\frac{1}{\beta}}\\
    &\leq C (1+|x|) \|\nabla f\|_{\infty}  \eta^{\frac{\beta-1}{\beta}}\eta^{1+\frac{1}{\alpha}}\\
    &=C (1+|x|) \|\nabla f\|_{\infty}\eta^{2+\frac{1}{\alpha}-\frac{1}{\beta}}.
\end{align*}
Hence, we have
\begin{align*}
    |\mathsf{I}|
    \leq C (1+|x|)\left(\|\nabla f\|_{\infty}+\|\nabla^{2}f\|_{\mathrm{HS},\infty}\right)\eta^{2+\frac{1}{\alpha}-\frac{1}{\beta}}.
\end{align*}

For $\mathsf{II}$ we use It\^{o}'s formula and the definitions \eqref{SDE}, \eqref{EM2} of $X_\eta^x$ and $Y_1^x$ to see
\begin{align*}
    \mathsf{II}
    &=\Ee\left\langle\nabla f\left(x+\eta b(x)\right),\int_{0}^{\eta}\left[b(X_{s}^{x})-b(x)\right] \dif s\right\rangle\\
    &=\left\langle\nabla f\left(x+\eta b(x)\right),\int_{0}^{\eta}\Ee\left[b(X_{s}^{x})-b(x)\right] \dif s\right\rangle\\
    &=\left\langle\nabla f\left(x+\eta b(x)\right),\int_{0}^{\eta}\int_{0}^{s}\Ee\left[
    \left\langle\nabla b(X_{r}^{x}),b(X_{r}^{x})\right\rangle
    + (- \Delta)^{\frac{\alpha}{2}} b(X_{r}^{x})\right] \dif r\,\dif s\right\rangle\\
    &\leq C\|\nabla f\|_{\infty} (1+|x|) \eta^{2}.
\end{align*}
In the last inequality we use the estimate \eqref{linear} (for $b(x_r^x)$) and Lemma~\ref{laplace} (for $(-\Delta)^{\alpha/2}b(X_r^x)$), combined with the moment estimate \eqref{2}.

Finally, $\mathsf{III}$ is estimated by Lemma \ref{holder2} and \eqref{3},
\begin{align*}
    \mathsf{III}
    &\leq C\left(\|\nabla f\|_{\infty}+\|\nabla^{2}f\|_{\mathrm{HS},\infty}\right) \Ee\left|X_{\eta}^{x}-Y_{1}^{x}\right|^{\beta}\\
    &\leq C (1+|x|^{\beta})  \left(\|\nabla f\|_{\infty}+\|\nabla^{2}f\|_{\mathrm{HS},\infty}\right)\eta^{\beta+\frac{\beta}{\alpha}}\\
    &\leq C(1+|x|^{\beta})\left(\|\nabla f\|_{\infty}+\|\nabla^{2}f\|_{\mathrm{HS},\infty}\right)\eta^{2}.
\end{align*}
This finishes the proof.
\end{proof}

\subsection{Proof of Theorems \ref{main} and \ref{main2}}

\begin{proof}[Proof of Theorem \ref{main}.]  Thanks to the discrete version of the classical Duhamel principle, it is easy to check that for $h\in\mathrm{Lip}(1)$
\begin{equation} \label{e:P-Q1}
    P_{N\eta}h(x)-\tilde{Q}_Nh(x)
    = \sum_{i=1}^{N}\tilde{Q}_{i-1}\big(P_{\eta}-\tilde{Q}_{1}\big)P_{(N-i)\eta}h(x).
\end{equation}
Then we have
\begin{equation}\label{definition2}
\begin{aligned}
    W_{1}\left(\mathrm{law}(X_{\eta N}),\mathrm{law}(\tilde{Y}_{N})\right)
        &=\sup_{h\in\mathrm{Lip}(1)}|P_{N\eta}h(x)-\tilde{Q}_Nh(x)|\\
        &\leq\sum_{i=1}^{N-1}\sup_{h\in\mathrm{Lip}(1)}\left|\tilde{Q}_{i-1}\big(P_{\eta}-\tilde{Q}_{1}\big) P_{(N-i)\eta}h(x)\right|\\
        &\quad \mbox{}+\sup_{h\in\mathrm{Lip}(1)}\left|\tilde{Q}_{N-1}(P_\eta-\tilde{Q}_1)h(x)\right|.
\end{aligned}
\end{equation}
First, we bound the last term. By \eqref{2} with $\beta=1$, \eqref{EM}
and \eqref{gra}, for $h\in\mathrm{Lip}(1)$ and $\eta<1$,
\begin{align*}
    \left|(P_\eta-\tilde{Q}_1)h(x)\right|
    &=\left|\Ee h(X_\eta^x) - \Ee h(\tilde{Y}_{1}) \right|\\
    &\leq\left|\Ee h(X_\eta^x) -h(x) \right| + \left| \Ee h(\tilde{Y}_{1})-h(x) \right|\\
    &\leq \Ee |X_\eta^x-x| + \Ee |\tilde{Y}_{1}-x|\\
    &\leq C(1+|x|)\,\eta^{1/\alpha} + \eta\,|b(x)|+\sigma^{-1}\eta^{1/\alpha} \Ee |\widetilde{Z}_1|\\
    &\leq C(1+|x|)\,\eta^{1/\alpha} + \eta^{1/\alpha}(|b(0)|+\theta_2|x|)+\sigma^{-1}\eta^{1/\alpha} \Ee |\widetilde{Z}_1|\\
    &\leq C(1+|x|)\,\eta^{1/\alpha}.
\end{align*}
Together with \eqref{momentd} we get
\begin{equation}\label{last}
    \sup_{h\in\mathrm{Lip}(1)}\left|\tilde{Q}_{N-1}(P_\eta-\tilde{Q}_1)h(x)\right|
    \leq C(1+\Ee |\tilde{Y}_{N-1}^{x}|)
    \eta^{1/\alpha}
    \leq C(1+|x|)\eta^{2/\alpha-1}.
\end{equation}
Next, we bound the first term in \eqref{definition2}; we distinguish between two cases:

\noindent
\textbf{Case 1:} $N\leq\eta^{-1}+1$. By Lemmas \ref{compare} and \ref{regular},
\begin{align*}
    \left|\big(P_{\eta}-\tilde{Q}_{1}\big) P_{(N-i)\eta}h(x)\right|
    &\leq C(1+|x|)\left(\|\nabla P_{(N-i)\eta}h\|_{\infty}+\|\nabla^{2}P_{(N-i)\eta}h\|_{\mathrm{HS},\infty}\right)\eta^{2/\alpha}\\
    &\leq C(1+|x|)[(N-i)\eta]^{-1/\alpha}\eta^{2/\alpha}.
\end{align*}
Combining this with \eqref{momentd}, we get
\begin{equation}\label{jterm}
\begin{aligned}
    \sup_{h\in\mathrm{Lip}(1)}\left|Q_{i-1}\big(P_{\eta}-\tilde{Q}_{1}\big)
    P_{(N-i)\eta}h(x)\right|
    &\leq C
    (1+\Ee |\tilde{Y}_{i-1}^{x}|)[(N-i)\eta]^{-1/\alpha}
    \eta^{2/\alpha}\\
    &\leq C
    (1+|x|)[(N-i)\eta]^{-1/\alpha}
    \eta^{2/\alpha}.
\end{aligned}
\end{equation}
Since $N-1\leq\eta^{-1}$,
\begin{align*}
    \sum_{i=1}^{N-1}[(N-i)\eta]^{-1/\alpha}=\eta^{-1/\alpha}\sum_{i=1}^{N-1}i^{-1/\alpha}&\leq\eta^{-1/\alpha}\int_0^{N-1}r^{-1/\alpha}\,\dif r\\
    &=\frac{\alpha}{\alpha-1}\,\eta^{-1/\alpha}(N-1)^{-1/\alpha+1}\leq \frac{\alpha}{\alpha-1}\,\eta^{-1}.
\end{align*}
This gives the upper bound
\begin{align*}
    \sum_{i=1}^{N-1}\sup_{h\in\mathrm{Lip}(1)}\left|\tilde{Q}_{i-1}\big(P_{\eta}-\tilde{Q}_{1}\big) P_{(N-i)\eta}h(x)\right|
    &\leq C(1+|x|)\eta^{2/\alpha} \sum_{i=1}^{N-1}[(N-i)\eta]^{-1/\alpha}\\
    &\leq C\frac{\alpha}{\alpha-1}\,(1+|x|)\eta^{2/\alpha-1}.
\end{align*}

\noindent
\textbf{Case 2:} $N>\eta^{-1}+1$. By Proposition \ref{W1.1}, for any $x,y\in \mathds{R}^{d}$, there exist constants $C>0$ and $\lambda>0$ such that
\begin{gather*}
    | P_{t}h(x)-P_{t}h(y)|\leq C\eup^{-\lambda t}|x-y|,
    \quad h\in\mathrm{Lip}(1),\; t\geq 0.
\end{gather*}
This implies that
\begin{align*}
    \sup_{h\in\mathrm{Lip}(1)}\left|Q_{i-1}
    \big(P_{\eta}-\tilde{Q}_{1}\big)
    P_{(N-i)\eta}h(x)\right|
    &=\sup_{h\in\mathrm{Lip}(1)}\left|\tilde{Q}_{i-1}
    \big(P_{\eta}-\tilde{Q}_{1}\big)P_1
    P_{(N-i)\eta-1}h(x)\right|\\
    &\leq C\eup^{-\lambda[(N-i)\eta-1]}
    \sup_{g\in\mathrm{Lip}(1)}\left|\tilde{Q}_{i-1}
    \big(P_{\eta}-\tilde{Q}_{1}\big)P_1
    g(x)\right|,
\end{align*}
where $i\leq\lfloor N-\eta^{-1} \rfloor$. By Lemmas \ref{compare} and \ref{regular},
\begin{gather*}
    \left|\big(P_{\eta}-\tilde{Q}_{1}\big) P_{1}g(x)\right|
   \leq C(1+|x|) \left(\|\nabla P_1g\|_{\infty}+\|\nabla^{2}P_1g\|_{\mathrm{HS},\infty}\right) \eta^{2/\alpha}
    \leq C(1+|x|)\,\eta^{2/\alpha}.
\end{gather*}
Combining this with \eqref{momentd}, we get
\begin{align*}
    \sum_{i=1}^{\lfloor N-\eta^{-1} \rfloor} \sup_{h\in\mathrm{Lip}(1)}\left|\tilde{Q}_{i-1}\big(P_{\eta}-\tilde{Q}_{1}\big) P_{(N-i)\eta}h(x)\right|
    &\leq C\eta^{2/\alpha} \sum_{i=1}^{\lfloor N-\eta^{-1} \rfloor} \eup^{-\lambda[(N-i)\eta-1]} \left(1+\Ee |\tilde{Y}_{i-1}^{x}|\right)\\
    &\leq C(1+|x|)\eta^{2/\alpha} \sum_{i=1}^{\lfloor N-\eta^{-1} \rfloor} \eup^{-\lambda[(N-i)\eta-1]}.
\end{align*}
Observe that
\begin{align*}
    \sum_{i=1}^{\lfloor N-\eta^{-1} \rfloor} \eup^{-\lambda[(N-i)\eta-1]}
    &=\sum_{i=\lfloor \eta^{-1} \rfloor}^{N-1} \eup^{-\lambda(i\eta-1)}
    \leq\eup^\lambda\int_{\lfloor \eta^{-1} \rfloor-1}^{N-1}\eup^{-\lambda\eta r}\,\dif r\\
    &\leq\eup^\lambda\eta^{-1}\int_0^\infty\eup^{-\lambda r}\,\dif r
    = \lambda^{-1}\eup^\lambda\eta^{-1}.
\end{align*}
Thus, we get
\begin{gather*}
    \sum_{i=1}^{\lfloor N-\eta^{-1} \rfloor}
    \sup_{h\in\mathrm{Lip}(1)}\left|\tilde{Q}_{i-1}\big(P_{\eta}-\tilde{Q}_{1}\big)
    P_{(N-i)\eta}h(x)\right|
    \leq C\lambda^{-1}\eup^\lambda
    (1+|x|)\eta^{2/\alpha-1}.
\end{gather*}
For $i\geq\lfloor N-\eta^{-1} \rfloor+1$, by almost the same as the calculation in the first case, we find
\begin{align*}
    \sum_{i=\lfloor N-\eta^{-1} \rfloor+1}^{N-1}[(N-i)\eta]^{-1/\alpha}
    \leq \frac{\alpha}{\alpha-1}\,\eta^{-1}.
\end{align*}
Combining this with \eqref{jterm}, we obtain
\begin{align*}
    &\sum_{i=\lfloor N-\eta^{-1} \rfloor+1}^{N-1}\sup_{h\in\mathrm{Lip}(1)}\left|\tilde{Q}_{i-1}\big(P_{\eta}-\tilde{Q}_{1}\big) P_{(N-i)\eta}h(x)\right|\\
    &\qquad\leq C (1+|x|)\,\eta^{2/\alpha} \sum_{i=\lfloor N-\eta^{-1} \rfloor+1}^{N-1}[(N-i)\eta]^{-1/\alpha}
    \leq C\frac{\alpha}{\alpha-1}\,(1+|x|)\,\eta^{2/\alpha-1}.
\end{align*}

We have just shown estimates for $\sum_{i=1}^{\lfloor N-\eta^{-1} \rfloor}\dots$ and $\sum_{i=\lfloor N-\eta^{-1} \rfloor+1}^{N-1}\dots$. Adding them up we arrive at
\begin{align*}
    &\sum_{i=1}^{N-1}\sup_{h\in\mathrm{Lip}(1)}\left|\tilde{Q}_{i-1}\big(P_{\eta}-\tilde{Q}_{1}\big) P_{(N-i)\eta}h(x)\right|\\
    =&\left(\sum_{i=1}^{\lfloor N-\eta^{-1} \rfloor} + \sum_{i=\lfloor N-\eta^{-1} \rfloor+1}^{N-1}\right)
    \sup_{h\in\mathrm{Lip}(1)}\left|\tilde{Q}_{i-1}\big(P_{\eta}-\tilde{Q}_{1}\big) P_{(N-i)\eta}h(x)\right|\leq C(1+|x|)\eta^{2/\alpha-1}.
\end{align*}

Both Case~1 and Case~2 lead to an estimate of the form
\begin{gather*}
    \sum_{i=1}^{N-1} \sup_{h\in\mathrm{Lip}(1)}\left|\tilde{Q}_{i-1}\big(P_{\eta}-\tilde{Q}_{1}\big) P_{(N-i)\eta}h(x)\right|
    \leq C(1+|x|)\,\eta^{2/\alpha-1}.
\end{gather*}
Substituting this and \eqref{last} into \eqref{definition2}, the first assertion of Theorem~\ref{main} follows.

\medskip
It remains to prove Part \eqref{main-ii}. It is easy to see from \eqref{exact1} and \eqref{exact2}, that we have
\begin{gather*}
     \lim_{k\to\infty}W_{1}\big(\mu,\mathrm{law}(X_{\eta k})\big)
     =\lim_{k\to\infty}W_{1}\big(\mathrm{law}(\tilde{Y}_{k}),\tilde{\mu}_{\eta}\big)
     =0.
\end{gather*}
By the triangle inequality and Part \eqref{main-i} with $x=0$,
\begin{align*}
     W_{1}\big(\mu,\tilde{\mu}_{\eta}\big)
     &\leq W_{1}\big(\mu,\mathrm{law}(X_{\eta N})\big)
        + W_{1}\big(\mathrm{law}(X_{\eta N}),\mathrm{law}(\tilde{Y}_{N})\big)
        + W_{1}\big(\mathrm{law}(\tilde{Y}_{N}),\tilde{\mu}_{\eta}\big)\\
     &\leq W_{1}\big(\mu,\mathrm{law}(X_{\eta N})\big)
        + C\,\eta^{2/\alpha-1}
        + W_{1}\big(\mathrm{law}(\tilde{Y}_{N}),\tilde{\mu}_{\eta}\big).
\end{align*}
Letting $N\to\infty$ finishes the proof.
\end{proof}

\begin{proof}[Proof of Theorem \ref{main2}.]
In the above proof, replacing Lemma \ref{compare} and \eqref{exact2} with Lemma \ref{taylor} and \eqref{exact22}, the proof of Theorem \ref{main2} is similar to that of Theorem \ref{main}.
\end{proof}

\section{Malliavin calculus and the proof of Lemma \ref{regular}}\label{A}

\subsection{Jacobi flow associated with the SDE \eqref{SDE}}

The \emph{Jacobian flow} is the derivative of $X_{t}^{x}$ with respect to the initial value $x$; the Jacobian flow
in direction $v\in\rd$ is defined by
\begin{align*}
    \nabla_{v}X_{t}^{x} := \lim_{\epsilon\to 0}\frac{X_{t}^{x+\epsilon v}-X_{t}^{x}}{\epsilon}, \quad t\geq 0.
\end{align*}
This limit exists and satisfies
\begin{align}\label{derivative}
    \frac{\dif}{\dif t}\nabla_{v}X_{t}^{x}=\nabla_{\nabla_{v}X_{t}^{x}}b(X_{t}^{x}), \quad \nabla_{v}X_{0}^{x}=v.
\end{align}
Similarly, for $v_{1},v_{2}\in\rd$, we can define $\nabla_{v_{2}}\nabla_{v_{1}}X_{t}^{x}$, which satisfies
\begin{align}\label{secderivative}
    \frac{\dif}{\dif t}\nabla_{v_{2}}\nabla_{v_{1}}X_{t}^{x}
    = \nabla_{\nabla_{v_{2}}\nabla_{v_{1}}X_{t}^{x}} b(X_{t}^{x})
        +\nabla_{\nabla_{v_{2}}X_{t}^{x}}\nabla_{\nabla_{v_{1}}X_{t}^{x}}b(X_{t}^{x}),
    \quad \nabla_{v_{1}}\nabla_{v_{1}}X_{0}^{x}=0.
\end{align}
Then, we first have the following estimates of $\nabla_{v_{1}}X_{t}^{x}$ and $\nabla_{v_{2}}\nabla_{v_{1}}X_{t}^{x}$.

\begin{lemma}
    For any starting point $x\in\rd$ and all directions $v_{1},v_{2}\in\rd$ the following \textup{(}deterministic\textup{)} estimates hold:
\begin{align}\label{firstbound}
    |\nabla_{v_{1}}X_{t}^{x}|
    &\leq \eup^{\theta_{2}}\left|v_{1}\right|, && t\in (0,1],
\\\label{secondbound}
    |\nabla_{v_{2}}\nabla_{v_{1}}X_{t}^{x}|
    & \leq\frac{\theta_{3}}{2\sqrt{2}\theta_{2}} \,\eup^{4\theta_{2}}\left|v_{1}\right|\left|v_{2}\right|, && t\in (0,1].
\end{align}
\end{lemma}

\begin{proof}
By \eqref{derivative} and \eqref{gra}, we have
\begin{align*}
    \frac{\dif}{\dif t}|\nabla_{v_{1}}X_{t}^{x}|^{2}
    = 2\left\langle\nabla_{v_{1}}X_{t}^{x},
    \nabla_{\nabla_{v_{1}}X_{t}^{x}} b(X_{t}^{x})\right\rangle
    \leq 2\theta_{2}|\nabla_{v_{1}}X_{t}^{x}|^2,
\end{align*}
and Gronwall's inequality yields for  $t\in(0,1]$
\begin{align*}
    |\nabla_{v_1}X_{t}^{x}|^{2}\leq \eup^{2\theta_{2}t}\left|v_{1}\right|^{2}
    \leq\eup^{2\theta_{2}}\left|v_{1}\right|^{2}.
\end{align*}
This proves the first assertion. Writing $\zeta(t):=\nabla_{v_{2}}\nabla_{v_{1}}X_{t}^{x}$, we see from \eqref{secderivative}, \eqref{gra}, \eqref{firstbound}, the Cauchy--Schwarz inequality, and the elementary estimate $2AB \leq A^2+B^2$ that
\begin{align*}
    \frac{\dif}{\dif t}|\zeta(t)|^{2}
    &= 2\left\langle\zeta(t),\nabla_{\zeta(t)} b(X_{t}^{x})\right\rangle + 2\left\langle\zeta(t),\nabla_{\nabla_{v_{2}}X_{t}^{x}}\nabla_{\nabla_{v_{1}}X_{t}^{x}}b(X_{t}^{x})\right\rangle\\
    &\leq 2\theta_{2}|\zeta(t)|^{2} + 2\theta_{3}\eup^{2\theta_{2}}\left|v_{1}\right|\left|v_{2}\right||\zeta(t)|\\
    &\leq 4\theta_{2}|\zeta(t)|^{2} + \frac{\theta_{3}^{2}}{2\theta_{2}}\,\eup^{4\theta_{2}}\left|v_{1}\right|^{2}\left|v_{2}\right|^{2}.
\end{align*}
Since $\zeta(0)=0$, we can use again Gronwall's inequality and get for all $t\in(0,1]$
\begin{equation*}
    |\zeta(t)|^{2}
    \leq \frac{\theta_{3}^{2}}{2\theta_{2}} \,\eup^{4\theta_{2}}\left|v_{1}\right|^{2}\left|v_{2}\right|^{2} \int_{0}^{t}\eup^{4\theta_{2}(t-s)}\,\dif s
    \leq \frac{\theta_{3}^{2}}{8\theta_{2}^{2}}\,\eup^{8\theta_{2}}\left|v_{1}\right|^{2}\left|v_{2}\right|^{2}.
\qedhere
\end{equation*}
\end{proof}

\subsection{Bismut's formula}\label{malliavin} (See also \cite{Nor86}).
Let $u\in L_{loc}^{2}([0,\infty)\times(\Omega,\mathcal{F},\mathds{P});\rd)$, i.e.\ we have $\Ee \int_{0}^{t}|u(s)|^{2}\,\dif s<\infty$ for all $t>0$.  Let $\{W_{t}\}_{t\geq 0}$ be a $d$-dimensional standard Brownian motion and assume that $u$ is adapted to the filtration $(\mathcal{F}_{t})_{t\geq 0}$ with $\mathcal{F}_{t}:=\sigma(W_{s}:0\leq s\leq t);$ i.e.\ $u(t)$ is $\mathcal{F}_{t}$ measurable for $t\geq 0$.  Define
\begin{align}\label{bismutf}
    U
    = \int_{0}^{\bullet}u(s)\,\dif s.
\end{align}
For a $t>0$, let $F_{t}:C([0,t],\rd)\to\real^m$
be an $\mathcal{F}_{t}$ measurable map. If the following limit exists
\begin{align*}
    D_{U}F_{t}(W)
    = \lim_{\epsilon\to 0}\frac{F_{t}(W+\epsilon U)-F_{t}(W)}{\epsilon}
\end{align*}
in $L^{2}((\Omega,\mathcal{F},\mathds{P});\real^m)$, then $F_{t}(W)$ is said to be $m$-dimensional \emph{Malliavin differentiable} and $D_{U}F_{t}(W)$ is called the \emph{Malliavin derivative} of $F_{t}(W)$ in the direction $U$.

Let $\phi\in \mathcal{C}_b^2(\rd,\real)$ and both $F_{t}(W)$ and $G_{t}(W)$  be $d$-dimensional Malliavin differentiable functionals. Then we have the following product and chain rules:
\begin{gather*}
    D_{U}\big(\langle F_{t}(W),G_{t}(W)\rangle\big)
    =\langle D_{U}F_{t}(W),G_{t}(W)\rangle
    +\langle F_{t}(W),D_{U}G_{t}(W) \rangle,
\intertext{and}
    D_{U}\nabla\phi(F_{t}(W))
    =\nabla_{D_{U}F_{t}(W)}\nabla\phi(F_{t}(W)).
\end{gather*}

The following integration by parts formula is often called \emph{Bismut's formula}. For a Malliavin differentiable $F_{t}(W)$ such that $F_{t}(W)$, $D_{U}F_{t}(W)\in L^{2}((\Omega,\mathcal{F},\mathds{P});\real)$,
we have
\begin{align}\label{bismut}
    \Ee [D_{U}F_{t}(W)]
    = \Ee \left[F_{t}(W)\int_{0}^{t}
    \left\langle u(s),\dif W_{s}
    \right\rangle
    \right].
\end{align}

\subsection{Time-change method for the SDE \eqref{SDE}}\label{timec} (See also \cite{Z}).
We will now turn to the SDE driven by a rotationally invariant $\alpha$-stable L\'evy process with $\alpha\in(1,2)$. We can express such drivers as subordinated Brownian motion. More precisely, let $\{S_{t}\}_{t\geq 0}$ be an independent $\frac{\alpha}{2}$-stable subordinator. Then, $Z_{t}:=W_{S_{t}}$ is a rotationally invariant $\alpha$-stable L\'evy process, see e.g.\ \cite{sato}. This means that we can re-write \eqref{SDE} as
\begin{align}\label{SDES}
    \dif X_{t} = b(X_{t})\,\dif t + \dif W_{S_{t}}, \quad X_{0}=x.
\end{align}

Let $\mathds{W}$ be the space of all continuous functions from $[0,\infty)$ to $\rd$ vanishing at $t=0$; we equip $\mathds{W}$ with the topology of locally uniform convergence, and the Wiener measure $\mu_{\mathds{W}}$; therefore, the coordinate process
\begin{align*}
    W_{t}(w)=w_{t}
\end{align*}
are a standard $d$-dimensional Brownian motion. Let $\mathds{S}$ be the space of all increasing, c\`adl\`ag (right continuous with finite left limits) functions from $[0,\infty)$ to $[0,\infty)$ vanishing at $t=0$; we equip $\mathds{S}$ with the Skorohod metric and the probability measure $\mu_{\mathds{S}}$ so that for any $l \in \mathds S$ the coordinate process
\begin{align*}
    S_{t}(l):=l_{t}
\end{align*}
is an $\frac{\alpha}{2}$-stable subordinator. On the product measure space
\begin{align*}
    (\Omega,\mathcal{F},\mathds{P})
    := \left(\mathds{W}\times\mathds{S}, \mathcal{B}(\mathds{W})\otimes\mathcal{B}(\mathds{S}),\mu_{\mathds{W}}\times\mu_{\mathds{S}}\right),
\end{align*}
we define
\begin{align*}
    L_{t}(w,l):=w_{l_{t}}.
\end{align*}
The process $\{L_{t}\}_{t\geq 0}$ is a rotationally invariant $\alpha$-stable L\'evy process on $(\Omega,\mathcal{F},\mathds{P})$.  We will use the following two natural filtrations associated with the L\'{e}vy process $L_{t}$ and the Brownian motion $W_{t}$:
\begin{align*}
    \mathcal{F}_{t}:=\sigma\left\{L_{s}(w,l); s\leq t\right\}
    \quad\text{and}\quad
    \mathcal{F}_{t}^{\mathds{W}}:=\sigma\left\{W_{s}(w); s\leq t\right\}.
\end{align*}
In particular, we can regard the solution $X_{t}^{x}$ of the SDE \eqref{SDES} as an $(\mathcal{F}_{t})$-adapted functional on $\Omega$, and therefore,
\begin{align*}
    \Ee f\left(X_{t}^{x}\right) = \int_{\mathds{S}}\int_{\mathds{W}} f\left(X_{t}^{x}(w_{l})\right) \mu_{\mathds{W}}(\dif w) \, \mu_{\mathds{S}}(\dif l).
\end{align*}
For every fixed $l\in\mathds{S}$, we denote by $X_{t}^{l}$ the solution to the SDE
\begin{align}\label{SDESL}
    \dif X_{t}^{l}
    = b(X_{t}^{l})\,\dif t + \dif W_{l_{t}}, \quad X_{0}^{l}=x.
\end{align}

We will now fix a path $l\in\mathds{S}$, and consider the SDE \eqref{SDESL}. Unless otherwise mentioned, all expectations are taken with respect to the Wiener space $\left(\mathds{W},\mathcal{B}(\mathds{W}),\mu_{\mathds{W}}\right)$. First of all, notice that $t\to W_{l_{t}}$ is a centered Gaussian process with independent increments. In particular, $W_{l_{t}}$ is a c\`adl\`ag $\mathcal{F}_{l_{t}}^{\mathds{W}}$-martingale. Thus, under \textbf{{Assumption A}}, it is well known that for each $x\in\rd$, the SDE \eqref{SDESL} admits a unique c\`adl\`ag $\mathcal{F}_{l_{t}}^{\mathds{W}}$-adapted solution $X_{t}^{x;l}$, see e.g.\ \cite[p.249, Theorem 6]{Pro04}.

The main aim of this section is to establish the following result:
\begin{lemma}\label{implement}
    Under \textbf{Assumption A} one has for all functions $\phi\in \mathcal{C}_{b}^{2}(\rd,\real)$, all directions $v_{1},v_{2}\in\rd$ and $x\in\rd$, $t\in(0,1]$
    \begin{equation}\label{grades}
        |\nabla_{v_1}X_{t}^{x;l}|
        \leq \eup^{\theta_2}\left|v_{1}\right|
    \end{equation}
    and
    \begin{align*}
    &\left|\Ee \left[\nabla_{\nabla_{v_{2}} X_{t}^{x;l}}\nabla_{\nabla_{v_{1}}X_{t}^{x;l}}
    \phi\bigl(X_{t}^{x;l}\bigr)\right]\right|\\
    &\quad\leq \left|\Ee \left[\frac{1}{l_{t}}
    \nabla_{\nabla_{v_{1}}X_{t}^{x;l}}\phi\bigl(X_{t}^{x;l}\bigr) \int_{0}^{t}\left\langle\nabla_{v_{2}}X_{s}^{x;l},
    \dif W_{l_{s}}\right\rangle\right]\right| + \|\nabla\phi\|_{\infty}\,\frac{\theta_{3}}{\sqrt{2\theta_{2}}}\,\eup^{2\theta_{2}}\left|v_{1}\right|\left|v_{2}\right|,
    \end{align*}
    where $\nabla_{v_{i}}X_{t}^{x;l}$ $(i=1,2)$ is determined by the following linear equation:
    \begin{align}\label{derivativel}
        \frac{\dif}{\dif t}\nabla_{v_{i}}X_{t}^{x;l}=\nabla_{\nabla_{v_{i}}X_{t}^{x;l}} b(X_{t}^{x;l}), \quad \nabla_{v_{i}}X_{0}^{x;l}=v_{i}.
    \end{align}
\end{lemma}

In order to prove Lemma \ref{implement}, we use a time-change argument to transform the SDE \eqref{SDESL} into an SDE driven by a standard Brownian motion; this allows us to use Bismut's formula \eqref{bismut}. For every $\epsilon\in(0,1)$ we define
\begin{align*}
    l_{t}^{\epsilon}
    := \frac{1}{\epsilon} \int_{t}^{t+\epsilon}l_{s}\,\dif s + \epsilon t
    = \int_{0}^{1}l_{\epsilon s+t}\,\dif s + \epsilon t.
\end{align*}
Since $t\mapsto l_{t}$ is increasing and right continuous, it follows that for each $t\geq 0$,
\begin{align*}
    l_{t}^{\epsilon}\downarrow l_{t} \quad\text{as}\quad \epsilon\downarrow 0.
\end{align*}
Moreover, $t\mapsto l_{t}^{\epsilon}$ is absolutely continuous and strictly increasing. Let $\gamma^{\epsilon}$ be the inverse function of $l^{\epsilon}$, i.e.
\begin{align*}
    l_{\gamma^{\epsilon}_{t}}^{\epsilon} = t, \quad t\geq l_{0}^{\epsilon}
    \quad\text{and}\quad
    \gamma_{l_{t}^{\epsilon}}^{\epsilon} = t, \quad t\geq 0.
\end{align*}
By definition, $\gamma_{t}^{\epsilon}$ is absolutely continuous on $[l_{0}^{\epsilon},\infty)$. Let $X_t^{x;l^\epsilon}$ be the solution to the SDE
\begin{equation}\label{xellepsilon}
    \dif X_{t}^{x;l^\epsilon}
    = b(X_{t}^{x;l^\epsilon})\,\dif t
    + \dif W_{l^\epsilon_{t}-l^\epsilon_{0}},
    \quad X_{0}^{x;l^\epsilon}=x.
\end{equation}
Let us now define
\begin{align*}
    Y_{t}^{x;l^{\epsilon}} := X_{\gamma_{t}^{\varepsilon}}^{x;l^{\epsilon}}, \quad t\geq l_{0}^{\epsilon}.
\end{align*}
Changing variables in \eqref{xellepsilon} we see that for $t\geq l_{0}^{\epsilon}$,
\begin{align}\label{SO}
    Y_{t}^{x;l^{\epsilon}}
    = x + \int_{0}^{\gamma_{t}^{\epsilon}} b\left(X_{s}^{x;l^{\epsilon}}\right)\,\dif s+W_{t-l^\epsilon_{0}}
    = x + \int_{l_{0}^{\epsilon}}^{t} b\left(Y_{s}^{x;l^{\epsilon}}\right) \dot{\gamma}_{s}^{\epsilon}\,\dif s+W_{t-l^\epsilon_{0}}
\end{align}
($\dot\gamma_s^\epsilon$ denotes the derivative in $s$). Hence, for any vector $v\in\rd$, we have
\begin{align}\label{Sfirsr}
    \nabla_{v}Y_{t}^{x;l^{\epsilon}}
    = v + \int_{l_{0}^{\epsilon}}^{t}\nabla_{\nabla_{v}Y_{s}^{x;l^{\epsilon}}} b\left(Y_{s}^{x;l^{\epsilon}}\right) \,\dot{\gamma}_{s}^{\epsilon}\,\dif s,
\end{align}
and the differential form can be written as
\begin{align*}
    \frac{\dif}{\dif t} \nabla_{v}Y_{t}^{x;l^{\epsilon}}
    = \nabla b\left(Y_{t}^{x;l^{\epsilon}}\right) \,\dot{\gamma}_{t}^{\epsilon} {\nabla_{v}Y_{t}^{x;l^{\epsilon}}},
    \quad t\geq l_{0}^{\epsilon},
\end{align*}
which has a solution of the form
\begin{align}\label{gradientSL}
    \nabla_{v}Y_{t}^{x;l^{\epsilon}} = J_{l_{0}^{\epsilon},t}^{x;l^{\epsilon}}v,
\end{align}
involving a matrix exponential
\begin{align}\label{Jst}
    J_{s,t}^{x;l^{\epsilon}}
    = \exp\left[\int_{s}^{t}\nabla b\left(Y_{s}^{x;l^{\epsilon}}\right) \dot{\gamma}_{s}^{\epsilon}\,\dif s\right],
    \quad l_{0}^{\epsilon}\leq s\leq t<\infty.
\end{align}
It is easy to see that $J_{s,t}^{x;l^{\epsilon}}J_{l_{0}^{\epsilon},s}^{x;l^{\epsilon}} = J_{l_{0}^{\epsilon},t}^{x;l^{\epsilon}}$ for all $l_{0}^{\epsilon}\leq s\leq t<\infty$.

Now, we come back to the Malliavin calculus from Section \ref{malliavin}. Fixing $t\geq l_{0}^{\epsilon}$ and $x\in\rd$, the solution $Y_{t}^{x;l^{\epsilon}}$ is a $d$-dimensional functional of Brownian motion $\{W_{s}\}_{l_{0}^{\epsilon}\leq s\leq t}$.

Let $U$ be as in Section \ref{malliavin}. The Malliavin derivative of $Y_{t}^{x;l^{\epsilon}}$ in direction $U$ exists in $L^{2}\left((\mathds{W},\mathcal{B}(\mathds{W}),\mu_{\mathds{W}});\rd\right)$ and is given by
\begin{align*}
    D_{U}Y_{t}^{x;l^{\epsilon}}(W)
    = \lim_{\delta\to 0}\frac{Y_{t}^{x;l^{\epsilon}}(W+\delta U)-Y_{t}^{x;l^{\epsilon}}(W)}{\delta}.
\end{align*}
To simplify notation, we drop the $W$ in $D_{U}Y_{t}^{x;l^{\epsilon}}(W)$ and write $D_{U}Y_{t}^{x;l^{\epsilon}}=D_{U}Y_{t}^{x;l^{\epsilon}}(W)$. By \eqref{SO}, it satisfies the equation
\begin{align*}
    D_{U}Y_{t}^{x;l^{\epsilon}}
    = \int_{l_{0}^{\epsilon}}^{t}\left(\nabla_{D_{U}Y_{s}^{x;l^{\epsilon}}} b\left(Y_{s}^{x;l^{\epsilon}}\right)\,\dot{\gamma}_{s}^{\epsilon}+u(s)\right) \dif s,
\end{align*}
the differential form of the above equation can be written as
\begin{align*}
    \frac{\dif}{\dif t}D_{U}Y_{t}^{x;l^{\epsilon}}
    = \nabla b\left(Y_{t}^{x;l^{\epsilon}}\right)\,\dot{\gamma}_{t}^{\epsilon} {D_{U}Y_{t}^{x;l^{\epsilon}}} + u(t),
    \quad t\geq l_{0}^{\epsilon},
\end{align*}
and this equation has a unique solution which is given via the matrix exponential \eqref{Jst}:
\begin{equation}\label{MgradientSL}
    D_{U}Y_{t}^{x;l^{\epsilon}}
    = \int_{l_{0}^{\epsilon}}^{t} J_{s,t}^{x;l^{\epsilon}}u(s)\,\dif s.
\end{equation}

For a fixed $t>0$, for any $v_1,v_2,x\in\rd$, we define $u_{i},U_{i}:[l_{0}^{\epsilon},t] \to \rd$ by
\begin{gather}\label{mark}
    u_{i}(s) := \frac{1}{t-l_{0}^{\epsilon}}\,\nabla_{v_i}Y_{s}^{x;l^{\epsilon}},
    \quad
    U_{i;s} := \int_0^su_{i}(r)\,\dif r
\end{gather}
for $l_{0}^{\epsilon}\leq s\leq t$ and $i=1,2$. Then
\begin{align}\label{equality}
    D_{U_{i}}Y_{s}^{x;l^{\epsilon}}
    = \frac{s-l_{0}^{\epsilon}}{t-l_{0}^{\epsilon}}\,\nabla_{v_i}Y_{s}^{x;l^{\epsilon}},
    \quad l_{0}^{\epsilon}\leq s\leq t.
\end{align}
In addition, \eqref{Sfirsr} implies that
for $s\in[l_{0}^{\epsilon},t]$
\begin{align}\label{SMalliavin}
    D_{U_{2}}\nabla_{v_{1}}Y_{s}^{x;l^{\epsilon}}
    =
    \int_{l_{0}^{\epsilon}}^{s}
    \left(\nabla_{D_{U_{2}} Y_{r}^{x;l^{\epsilon}}}\nabla_{\nabla_{v_{1}}Y_{r}^{x;l^{\epsilon}}} b\left(Y_{r}^{x;l^{\epsilon}}\right)+\nabla_{D_{U_{2}}\nabla_{v_{1}}Y_{r}^{x;l^{\epsilon}}} b\left(Y_{r}^{x;l^{\epsilon}}\right) \right)
    \dot{\gamma}_{r}^{\epsilon}\,\dif r.
\end{align}

The following lemma contains the upper bounds on the derivatives.
\begin{lemma}\label{suppleSLS}
    Let $v_1,v_2,x\in\rd$ and $t\in(0,1]$. Then,
    \begin{gather}\label{SLgra}
        |\nabla_{v_{i}}Y_{s}^{x;l^{\epsilon}}|
        \leq \eup^{\theta_{2}\gamma_{s}^{\epsilon}}\,|v_{i}|
    \intertext{and}\label{SLS}
        |D_{U_{2}}\nabla_{v_{1}}Y_{s}^{x;l^{\epsilon}}|
        \leq \frac{\theta_{3}}{\sqrt{2\theta_{2}}}\,\eup^{2\theta_{2}\gamma_{s}^{\epsilon}} \sqrt{\gamma_{s}^{\epsilon}}\,\left|v_{1}\right|\left|v_{2}\right|.
\end{gather}
    for $l_{0}^{\epsilon}\leq s\leq t$ and $i=1,2$.
\end{lemma}

\begin{proof}
Recall that $\theta_{1}>0$ and $\dot{\gamma}_{s}^{\epsilon}\geq 0$. By \eqref{Sfirsr} and \eqref{gra}, we have for any $l_{0}^{\epsilon}\leq s\leq t$
\begin{gather*}
    \frac{\dif}{\dif s}|\nabla_{v_{i}}Y_{s}^{x;l^{\epsilon}}|^{2}
    = 2\dot{\gamma}_{s}^{\epsilon} \left\langle\nabla_{v_{i}}Y_{s}^{x;l^{\epsilon}},\:\nabla_{\nabla_{v_{i}}Y_{s}^{x;l^{\epsilon}}} b\left(Y_{s}^{x;l^{\epsilon}}\right)\right\rangle
    \leq 2\theta_{2}\dot{\gamma}_{s}^{\epsilon}|\nabla_{v_{i}}Y_{s}^{x;l^{\epsilon}}|^{2},
\intertext{and this implies, because of Gronwall's lemma,}
    |\nabla_{v_{i}}Y_{s}^{x;l^{\epsilon}}|^{2}
    \leq \exp\left[2\theta_{2} \int_{l_{0}^{\epsilon}}^{s}\dot{\gamma}_{r}^{\epsilon}\,\dif r\right]|v_{i}|^{2}
    = \eup^{2\theta_{2}(\gamma_{s}^{\epsilon} - \gamma_{l_{0}^{\epsilon}}^{\epsilon})}|v_{i}|^{2}
    = \eup^{2\theta_{2}\gamma_{s}^{\epsilon}}|v_{i}|^{2}.
\end{gather*}

Using \eqref{SMalliavin} and \eqref{gra} we find for any $l_{0}^{\epsilon}\leq s\leq t$
\begin{align*}
    \frac{\dif}{\dif s}|D_{U_{2}} \nabla_{v_{1}}Y_{s}^{x;l^{\epsilon}}|^{2}
    &= 2\dot{\gamma}_{s}^{\epsilon} \left\langle D_{U_{2}}\nabla_{v_{1}}Y_{s}^{x;l^{\epsilon}},\: \nabla_{D_{U_{2}}\nabla_{v_{1}}Y_{s}^{x;l^{\epsilon}}} b\left(Y_{s}^{x;l^{\epsilon}}\right) \right\rangle\\
    &\quad\mbox{} + 2\dot{\gamma}_{s}^{\epsilon} \left\langle D_{U_{2}}\nabla_{v_{1}}Y_{s}^{x;l^{\epsilon}},\: \nabla_{D_{U_{2}}Y_{s}^{x;l^{\epsilon}}}\nabla_{\nabla_{v_{1}}Y_{s}^{x;l^{\epsilon}}}b\left(Y_{s}^{x;l^{\epsilon}}\right)\right\rangle\\
    &\leq 2\theta_{2}\dot{\gamma}_{s}^{\epsilon} \left|D_{U_{2}} \nabla_{v_{1}}Y_{s}^{x;l^{\epsilon}}\right|^{2}
        + 2\theta_{3}\dot{\gamma}_{s}^{\epsilon} \left|D_{U_{2}} \nabla_{v_{1}}Y_{s}^{x;l^{\epsilon}}\right| \left|D_{U_{2}} Y_{s}^{x;l^{\epsilon}}\right| \left|\nabla_{v_{1}}Y_{s}^{x;l^{\epsilon}}\right|.
\end{align*}
Now \eqref{equality} and \eqref{SLgra} imply
\begin{align*}
    \frac{\dif}{\dif s}\left|D_{U_{2}} \nabla_{v_{1}}Y_{s}^{x;l^{\epsilon}}\right|^{2}
    &\leq 2\theta_{2}\dot{\gamma}_{s}^{\epsilon} \left|D_{U_{2}}\nabla_{v_{1}}Y_{s}^{x;l^{\epsilon}}\right|^{2}
        + 2\theta_{3}\dot{\gamma}_{s}^{\epsilon} \left|D_{U_{2}}\nabla_{v_{1}}Y_{s}^{x;l^{\epsilon}}\right| \left|\nabla_{v_{2}}Y_{s}^{x;l^{\epsilon}}\right|\left|\nabla_{v_{1}}Y_{s}^{x;l^{\epsilon}}\right|\\
    &\leq 4\theta_{2}\dot{\gamma}_{s}^{\epsilon} \left|D_{U_{2}}\nabla_{v_{1}}Y_{s}^{x;l^{\epsilon}}\right|^{2}
        + \frac{\theta_{3}^{2}}{2\theta_{2}}\dot{\gamma}_{s}^{\epsilon}\left|\nabla_{v_{2}}Y_{s}^{x;l^{\epsilon}}\right|^{2}
        \left|\nabla_{v_{1}}Y_{s}^{x;l^{\epsilon}}\right|^{2}\\
    &\leq 4\theta_{2}\dot{\gamma}_{s}^{\epsilon}\left|D_{U_{2}}\nabla_{v_{1}}Y_{s}^{x;l^{\epsilon}}\right|^{2}
        + \frac{\theta_{3}^{2}}{2\theta_{2}} \dot{\gamma}_{s}^{\epsilon}\eup^{4\theta_{2} \gamma_{s}^{\epsilon}}\left|v_{1}\right|^{2}\left|v_{2}\right|^{2}.
\end{align*}
Since $D_{U_{2}}\nabla_{v_{1}}Y_{l_{0}^{\epsilon}}^{x;l^{\epsilon}}=0$, we can use Gronwall's inequality to see
\begin{equation*}
    \left|D_{U_{2}}\nabla_{v_{1}}Y_{s}^{x;l^{\epsilon}}\right|^{2}
    \leq \frac{\theta_{3}^{2}}{2\theta_{2}}\left|v_{1}\right|^{2}\left|v_{2}\right|^{2} \int_{l_{0}^{\epsilon}}^{s} \dot{\gamma}_{u}^{\epsilon} \eup^{4\theta_{2}\gamma_{u}^{\epsilon}}
    \eup^{4\theta_{2}(\gamma_{s}^{\epsilon} - \gamma_{u}^{\epsilon})}\,\dif u
    = \frac{\theta_{3}^{2}}{2\theta_{2}}
    |v_{1}|^{2}|v_{2}|^{2} \eup^{4\theta_{2}\gamma_{s}^{\epsilon}}\gamma_{s}^{\epsilon}.
\qedhere
\end{equation*}
\end{proof}

\begin{proof}[\textbf{Proof of Lemma \ref{implement}.}]
From \eqref{gra} we see that
\begin{align*}
    \frac{\dif}{\dif t}\left|\nabla_{v_{1}}X_{t}^{x;l}\right|^{2}
    = 2\left\langle\nabla_{v_{1}}X_{t}^{x;l},\:\nabla_{\nabla_{v_{1}}X_{t}^{x;l}} b(X_{t}^{x;l})\right\rangle
    \leq 2\theta_{2}\left|\nabla_{v_{1}}X_{t}^{x;l}\right|^2.
\end{align*}
This yields for all $t\in(0,1]$
\begin{align*}
    \left|\nabla_{v_1}X_{t}^{x;l}\right|^{2}
    \leq \eup^{2\theta_{2}t}\left|v_{1}\right|^{2}
    \leq \eup^{2\theta_{2}} \left|v_{1}\right|^{2},
\end{align*}
i.e.\ \eqref{grades} holds.

Using \eqref{equality} with $s=t$ and $i=2$, the chain rule and Bismut's formula \eqref{bismut}, we see that
\begin{align*}
    &\Ee \left[\nabla_{\nabla_{v_{2}}Y_{t}^{x;l^{\epsilon}}}\nabla_{\nabla_{v_{1}}Y_{t}^{x;l^{\epsilon}}}\phi\bigl(Y_{t}^{x;l^{\epsilon}}\bigr)\right]\\
    &\quad =\Ee \left[\nabla_{D_{U_{2}}Y_{t}^{x;l^{\epsilon}}}\nabla_{\nabla_{v_{1}}Y_{t}^{x;l^{\epsilon}}}\phi\bigl(Y_{t}^{x;l^{\epsilon}}\bigr)\right]\\
    &\quad =\Ee \left[D_{U_{2}} \left(\nabla_{\nabla_{v_{1}}Y_{t}^{x;l^{\epsilon}}}\phi\bigl(Y_{t}^{x;l^{\epsilon}}\bigr)\right)\right]
            - \Ee \left[\nabla_{D_{U_{2}}\nabla_{v_{1}}Y_{t}^{x;l^{\epsilon}}}\phi\bigl(Y_{t}^{x;l^{\epsilon}}\bigr)\right]\\
    &\quad =\frac{1}{t-l_{0}^{\epsilon}}\Ee  \left[\nabla_{\nabla_{v_{1}}Y_{t}^{x;l^{\epsilon}}}\phi\bigl(Y_{t}^{x;l^{\epsilon}}\bigr) \int_{0}^{t}\left\langle\nabla_{v_{2}} Y_{s}^{x;l^{\epsilon}},\dif W_{s}\right\rangle\right]
    -\Ee \left[\nabla_{D_{U_{2}}
    \nabla_{v_{1}}Y_{t}^{x;l^{\epsilon}}}\phi \bigl(Y_{t}^{x;l^{\epsilon}}\bigr)\right].
\end{align*}
We can now use the fact that for each $t\geq 0$,
\begin{align*}
    Y_{l_{t}^{\epsilon}}^{x;l^{\epsilon}} = X_{t}^{x;l^{\epsilon}}
    \quad\text{and}\quad
    \nabla_{v}Y_{l_{t}^{\epsilon}}^{x;l^{\epsilon}} = \nabla_{v}X_{t}^{x;l^{\epsilon}}.
\end{align*}
Replacing $t$ with $l_{t}^{\epsilon}$ in \eqref{mark}, this yields
\begin{align*}
    &\Ee \left[\nabla_{\nabla_{v_{2}}X_{t}^{x;l^{\epsilon}}}\nabla_{\nabla_{v_{1}}X_{t}^{x;l^{\epsilon}}}\phi\bigl(X_{t}^{x;l^{\epsilon}}\bigr)\right]\\
    &\quad =\frac{1}{l_{t}^{\epsilon}-l_{0}^{\epsilon}}\, \Ee \left[\nabla_{\nabla_{v_{1}}X_{t}^{x;l^{\epsilon}}}
    \phi\bigl(X_{t}^{x;l^{\epsilon}}\bigr)
    \int_{0}^{l_{t}^{\epsilon}}\left\langle\nabla_{v_{2}} Y_{s}^{x;l^{\epsilon}},\dif W_{s}\right\rangle\right]
    -\Ee \left[\nabla_{D_{U_{2}}\nabla_{v_{1}}Y_{l_{t}^{\epsilon}}^{x;l^{\epsilon}}}\phi\bigl(X_{t}^{x;l^{\epsilon}}\bigr)\right]\\
    &\quad =\frac{1}{l_{t}^{\epsilon}-l_{0}^{\epsilon}}\, \Ee \left[\nabla_{\nabla_{v_{1}}
    X_{t}^{x;l^{\epsilon}}}\phi\bigl(X_{t}^{x;l^{\epsilon}}\bigr)
    \int_{0}^{t}\left\langle\nabla_{v_{2}}X_{s}^{x;l^{\epsilon}}
    ,\dif W_{l_{s}^{\epsilon}}\right\rangle\right]
    -\Ee \left[\nabla_{D_{U_{2}}\nabla_{v_{1}}Y_{l_{t}^{\epsilon}}^{x;l^{\epsilon}}}\phi\left(X_{t}^{x;l^{\epsilon}} \right)\right].
\end{align*}
Since $\gamma_{l_{t}^{\epsilon}}^{\epsilon}=t$, \eqref{SLS} implies that for every $t\in(0,1]$
\begin{align*}
    \left|\Ee \left[\nabla_{D_{U_{2}}\nabla_{v_{1}}Y_{l_{t}^{\epsilon}}^{x;l^{\epsilon}}}\phi\bigl(X_{t}^{x;l^{\epsilon}}\bigr)\right]\right|
    &=\left|\Ee \left[\left\langle\nabla\phi\bigl(X_{t}^{x;l^{\epsilon}}\bigr),D_{U_{2}}\nabla_{v_{1}}Y_{l_{t}^{\epsilon}}^{x;l^{\epsilon}}\right\rangle\right]\right|\\
    &\leq \left\|\nabla\phi\right\|_{\infty} \frac{\theta_{3}}{\sqrt{2\theta_{2}}}\, \eup^{2\theta_{2}\gamma_{l_{t}^{\epsilon}}^{\epsilon}} \sqrt{\gamma_{l_{t}^{\epsilon}}^{\epsilon}} \left|v_{1}\right| \left|v_{2}\right|\\
    &= \left\|\nabla\phi\right\|_{\infty} \frac{\theta_{3}}{\sqrt{2\theta_{2}}}\, \eup^{2\theta_{2}t} \sqrt{t} \left|v_{1}\right| \left|v_{2}\right|\\
    &\leq \left\|\nabla\phi\right\|_{\infty} \frac{\theta_{3}}{\sqrt{2\theta_{2}}} \,\eup^{2\theta_{2}} \left|v_{1}\right| \left|v_{2}\right|,
\end{align*}
and so
\begin{equation}\label{upperepsilon}
\begin{aligned}
    &\left|\Ee \left[\nabla_{\nabla_{v_{2}}X_{t}^{x;l^{\epsilon}}}\nabla_{\nabla_{v_{1}}X_{t}^{x;l^{\epsilon}}}\phi \bigl(X_{t}^{x;l^{\epsilon}}\bigr)\right]\right|\\
    &\leq\left|\frac{1}{l_{t}^{\epsilon}-l_{0}^{\epsilon}}\, \Ee \left[\nabla_{\nabla_{v_{1}}X_{t}^{x;l^{\epsilon}}}
    \phi\bigl(X_{t}^{x;l^{\epsilon}}\bigr)
    \int_{0}^{t}
    \left\langle\nabla_{v_{2}}X_{s}^{x;l^{\epsilon}},\dif W_{l_{s}^{\epsilon}}\right\rangle\right]\right|
    + \left\|\nabla\phi\right\|_{\infty} \frac{\theta_{3}}{\sqrt{2\theta_{2}}} \,\eup^{2\theta_{2}} \left|v_{1}\right| \left|v_{2}\right|.
\end{aligned}
\end{equation}
By the same argument as in the proof of \cite[Lemma 2.5]{Z}, we obtain
\begin{equation}\label{11upperepsilon}
\begin{aligned}
    &\lim_{\epsilon\to 0}\frac{1}{l_{t}^{\epsilon}-l_{0}^{\epsilon}}
    \Ee \left[\nabla_{\nabla_{v_{1}}X_{t}^{x;l^{\epsilon}}}
    \phi\bigl(X_{t}^{x;l^{\epsilon}}\bigr)
    \int_{0}^{t}\left\langle
    \nabla_{v_{2}}X_{s}^{x;l^{\epsilon}},\dif W_{l_{s}^{\epsilon}}\right\rangle\right]\\
    &\quad=\frac{1}{l_{t}}\,\Ee  \left[\nabla_{\nabla_{v_{1}}X_{t}^{x;l}}
    \phi\left(X_{t}^{x;l}\right)
    \int_{0}^{t}\left\langle\nabla_{v_{2}}X_{s}^{x;l},\dif W_{l_{s}}\right\rangle\right].
\end{aligned}
\end{equation}
On the other hand, from \cite[Lemma 2.2]{Z}, we know that
\begin{equation}\label{22upperepsilon}
\begin{aligned}
    \lim_{\epsilon\to 0}\Ee  \left[\nabla_{\nabla_{v_{2}}X_{t}^{x;l^{\epsilon}}}\nabla_{\nabla_{v_{1}} X_{t}^{x;l^{\epsilon}}}\phi\bigl(X_{t}^{x;l^{\epsilon}}\bigr)\right]
    = \Ee \left[\nabla_{\nabla_{v_{2}}X_{t}^{x;l}}\nabla_{\nabla_{v_{1}}X_{t}^{x;l}}\phi\left(X_{t}^{x;l}\right)\right].
\end{aligned}
\end{equation}
Letting in \eqref{upperepsilon} $\epsilon\to 0$ and using \eqref{11upperepsilon} and \eqref{22upperepsilon}, completes the proof.
\end{proof}

\subsection{Proof of Lemma~\ref{regular}}\label{suppleimp}
Because of \eqref{firstbound}, we can use the differentiation theorem for parameter dependent integrals to get
\begin{align*}
    \left|\nabla_{v}P_{t}h(x)\right|
    = \left|\nabla_{v}\Ee [h(X_{t}^{x})]\right|
    =& \left|\Ee \left[\nabla_{\nabla_{v}X_{t}^{x}} h(X_{t}^{x})\right]\right|\\
    =&\left|\Ee \left[\left\langle\nabla h(X_{t}^{x}),\nabla_{v}X_{t}^{x}\right\rangle\right]\right|
    \leq \left\|\nabla h\right\| \eup^{\theta_{2}} \left|v\right|
    \leq \eup^{\theta_{2}}\left|v\right|.
\end{align*}

In order to see the second inequality, we define for every $\epsilon>0$
\begin{align}\label{ma2}
    h_{\epsilon}(x) := \int_{\rd} g_\epsilon(y)h(x-y)\,\dif y,
\end{align}
where $g_\epsilon$ is the density of the normal distribution $N(0,\epsilon^{2}I_{d})$.  It is easy to see that $h_{\epsilon}$ is smooth, $\lim_{\epsilon\to 0}h_{\epsilon}(x)=h(x)$, $\lim_{\epsilon\to 0}\nabla h_{\epsilon}(x)=\nabla h(x)$ and $|h_{\epsilon}(x)|\leq C(1+|x|)$ for all $x\in\rd$ and some $C>0$.  Moreover, $\|\nabla h_{\epsilon}\|\leq\|\nabla h\|\leq1$.  Using the differentiability theorem for parameter dependent integrals we get
\begin{align*}
    \nabla_{v_{2}}\nabla_{v_{1}}\Ee \left[h_{\epsilon}(X_{t}^{x})\right]
    = \Ee \left[\nabla_{\nabla_{v_{2}}\nabla_{v_{1}}X_{t}^{x}} h_{\epsilon}(X_{t}^{x})\right] + \Ee \left[\nabla_{\nabla_{v_{2}}X_{t}^{x}}\nabla_{\nabla_{v_{1}}X_{t}^{x}}h_{\epsilon}(X_{t}^{x})\right].
\end{align*}
From \eqref{secondbound} we get
\begin{align*}
    \left|\Ee \left[\nabla_{\nabla_{v_{2}} \nabla_{v_{1}}X_{t}^{x}} h_{\epsilon}(X_{t}^{x})\right]\right|
    =\left|\Ee \left[\left\langle\nabla h_{\epsilon}(X_{t}^{x}),\nabla_{v_{2}} \nabla_{v_{1}}X_{t}^{x}\right\rangle\right]\right|
    \leq \frac{\theta_{3}}{2\sqrt{2}\theta_{2}}\,\eup^{4\theta_{2}}\left|v_{1}\right|\left|v_{2}\right|.
\end{align*}
It follows from Lemma \ref{implement}, \cite[(3.3)]{Z} and \eqref{firstbound} that
\begin{align*}
    &\left|\Ee \big[\nabla_{\nabla_{v_{2}}X_{t}^{x}}\nabla_{\nabla_{v_{1}}X_{t}^{x}}h_{\epsilon}(X_{t}^{x})\big]\right|\\
    &\quad\leq \left|\Ee \left[\frac{1}{S_{t}}\nabla_{\nabla_{v_{1}}X_{t}^{x}} h_{\epsilon}\left(X_{t}^{x}\right)
    \int_{0}^{t}\left\langle\nabla_{v_{2}}X_{s}^{x},\dif W_{S_{s}}\right\rangle\right]\right|
    + \|\nabla h_{\epsilon}\|_{\infty}\frac{\theta_{3}}{\sqrt{2\theta_{2}}} \,\eup^{2\theta_{2}}\left|v_{1}\right|\left|v_{2}\right|\\
    &\quad\leq \eup^{\theta_{2}}\left|v_{1}\right|\Ee  \left[\frac{1}{S_{t}}\left| \int_{0}^{t}\left\langle\nabla_{v_{2}}X_{s}^{x},\dif W_{S_{s}}\right\rangle\right|\right]
    +\frac{\theta_{3}}{\sqrt{2\theta_{2}}}\,
    \eup^{2\theta_{2}}\left|v_{1}\right|\left|v_{2}\right|.
\end{align*}
The Cauchy--Schwarz inequality, It\^o's isometry and \eqref{grades} give
\begin{align*}
    \Ee \left[\frac{1}{S_{t}}\left| \int_{0}^{t}\left\langle\nabla_{v_{2}}X_{s}^{x},\dif W_{S_{s}}\right\rangle\right|\right]
    &= \int_{\mathds{S}}\frac{1}{l_t}\,\Ee
    \left|\int_{0}^{t}\left\langle\nabla_{v_{2}}X_{s}^{x;l},\dif W_{l_{s}}\right\rangle\right| \mu_{\mathds{S}}(\dif l)\\
    &\leq \int_{\mathds{S}}\frac{1}{l_t}\left(\Ee  \int_{0}^{t}|\nabla_{v_{2}}X_{s}^{x;l}|^2\,\dif l_{s}\right)^{1/2}\,\mu_{\mathds{S}}(\dif l)\\
    &\leq \eup^{\theta_2}\left|v_{2}\right|\int_{\mathds{S}} \frac{1}{\sqrt{l_t}}\,\mu_{\mathds{S}}(\dif l)\\
    &= \eup^{\theta_2}\left|v_{2}\right|\Ee \left[ S_t^{-1/2} \right]\\
    &\leq C\eup^{\theta_2}\left|v_{2}\right| t^{-1/\alpha},
\end{align*}
where the last inequality is taken from \cite[Theorem 2.1\,(ii)\,(c)]{DSS17}. Thus, we have for all $t\in(0,1]$,
\begin{align*}
    \left|\nabla_{v_{2}}\nabla_{v_{1}}\Ee \left[h_{\epsilon}(X_{t}^{x})\right]\right|
    &\leq \left|\Ee \left[\nabla_{\nabla_{v_{2}} \nabla_{v_{1}}X_{t}^{x}}h_{\epsilon}(X_{t}^{x})\right]\right|
    + \left|\Ee \left[\nabla_{\nabla_{v_{2}}X_{t}^{x}}\nabla_{\nabla_{v_{1}}X_{t}^{x}}h_{\epsilon}(X_{t}^{x})\right]\right|\\
    &\leq \frac{\theta_{3}}{2\sqrt{2}\theta_{2}}\,\eup^{4\theta_{2}}\left|v_{1}\right|\left|v_{2}\right|
    + \eup^{\theta_{2}}\left|v_{1}\right|C\eup^{\theta_{2}}\left|v_{2}\right|t^{-1/\alpha}
    + \frac{\theta_{3}}{\sqrt{2\theta_{2}}}\,\eup^{2\theta_{2}}\left|v_{1}\right|\left|v_{2}\right|\\
    &\leq \eup^{4\theta_2}\left( \frac{\theta_{3}}{2\sqrt{2}\theta_{2}}+C t^{-1/\alpha}+\frac{\theta_{3}}{\sqrt{2\theta_{2}}}\right)\left|v_{1}\right|\left|v_{2}\right|\\
    &\leq C\eup^{4\theta_2}t^{-1/\alpha}\left|v_{1}\right|\left|v_{2}\right|.
\end{align*}
Finally, we can let $\epsilon\to 0$ using dominated convergence,
\begin{gather*}
    \lim_{\epsilon\to 0}\nabla_{v_{2}}\nabla_{v_{1}} \Ee \left[h_{\epsilon}(X_{t}^{x})\right]
    = \nabla_{v_{2}}\nabla_{v_{1}}\Ee \left[h(X_{t}^{x})\right],
\end{gather*}
completing the proof of Lemma~\ref{regular}.
\qed

\begin{appendix}

\section{Proofs of Propositions \ref{W1.1-0}, \ref{W1.1} and \ref{W1.1-1}, and Lemma \ref{E-Mmoment}}\label{sec2}
\begin{proof}[Proof of Proposition \ref{W1.1-0}]
The generator $\mathcal{L}^{\alpha}$ of the process $X_{t}$ is given
\begin{align*}
    \mathcal{L}^{\alpha}f(x)=\left\langle b(x),\nabla f(x)\right\rangle+(-\Delta)^{\alpha/2}f(x),
    \quad f\in\mathcal{C}^{2}(\rd,\real)
\end{align*}
where $(-\Delta)^{\alpha/2}$ is the fractional Laplace operator, which is the generator of the rotationally symmetric $\alpha$-stable L\'evy process $Z_{t}$; it is defined as a principal value integral
\begin{align}\label{frac}
    (-\Delta)^{\alpha/2}f(x)
    = C_{d,\alpha}\cdot\textrm{p.v.}\int_{\rd} \left(f(x+y)-f(x)\right) \frac{\dif y}{|y|^{\alpha+d}}.
\end{align}
It is not difficult to see that for all functions from the set
\begin{align*}
    \mathcal{D}
    :=
    \left\{f\in \mathcal{C}^{2}(\rd,\real)\,;\,\int_{|z|\geq1}\left| f(x+z)-f(x)\right| \frac{\dif z}{|z|^{\alpha+d}}<\infty\right\}.
\end{align*}
$\mathcal{L}^\alpha f$ is well-defined; moreover $\mathcal{D}\times \mathcal{L}^\alpha(\mathcal{D})$ can be embedded into the \emph{full generator} $\widehat{\mathcal{L}^{\alpha}}$, i.e.\ the set of all pairs of (bounded) Borel functions $(f,g)$ such that $f(X_t) - f(X_0) - \int_0^t g(X_s)\,\dif s$ is a (local) martingale, see the discussion in \cite[pp.~24--26]{BSW}.

Recall that $V_{\beta}(x)=(1+|x|^{2})^{\beta/2}$. It is easy to check that $V_{\beta}\in\mathcal{D}(\mathcal{L}^{\alpha})$. Since
\begin{gather*}
    \nabla V_{\beta}(x)
    = \frac{\beta x}{(1+|x|^{2})^{\frac{2-\beta}{2}}},
    \quad
    \nabla^{2}V_{\beta}(x)
    = \frac{\beta I_d}{(1+|x|^{2})^{1-\frac{\beta}{2}}}
    + \frac{\beta(\beta-2) xx^{\top}}{(1+|x|^{2})^{2-\frac{\beta}{2}}},
\end{gather*}
($I_d$ denotes the $d \times d$ identity matrix) we see that for any $x\in\rd$
\begin{gather*}
    |\nabla V_{\beta}(x)|\leq\beta|x|^{\beta-1},\quad
    \|\nabla^{2}V_{\beta}(x)\|_{\mathrm{HS}}
    \leq \beta(3-\beta)\sqrt{d}.
\end{gather*}
Thus, \eqref{diss} and Young's inequality ($AB \leq \frac 1p A^p + \frac 1q B^q$ with $p=\beta$ and $q = \beta/(\beta-1)$) imply
\begin{align*}
    \left\langle b(x),\nabla V_{\beta}(x)\right\rangle
    &=\frac{\beta}{(1+|x|^{2})^{\frac{2-\beta}{2}}}\left\langle b(x)-b(0),x\right\rangle+\frac{\beta}{(1+|x|^{2})^{\frac{2-\beta}{2}}}\left\langle b(0),x\right\rangle\\
    &\leq-\theta_{1}\frac{\beta|x|^{2}}{(1+|x|^{2})^{\frac{2-\beta}{2}}} + \frac{\beta K}{(1+|x|^{2})^{\frac{2-\beta}{2}}} + \frac{\beta|b(0)||x|}{(1+|x|^{2})^{\frac{2-\beta}{2}}}\\
    &\leq-\theta_{1}\beta V_{\beta}(x) + \theta_{1}\beta+\beta K+\beta|b(0)||x|^{\beta-1}\\
    &\leq-\theta_{1}V_{\beta}(x) + \theta_{1}\beta+\beta K+\theta_{1}^{1-\beta}|b(0)|^{\beta}.
\end{align*}
Therefore, we see from \eqref{frac} that
\begin{equation}\label{fracbound}
\begin{aligned}
    (-\Delta)^{\alpha/2}V_{\beta}(x)
    &=C_{d,\alpha}\int \left( V_{\beta}(x+y)-V_{\beta}(x)-\left\langle\nabla V_{\beta}(x),y\right\rangle \II_{(0,1)}(|y|)\right)\frac{\dif y}{|y|^{\alpha+d}}\\
    &=C_{d,\alpha}\int_{|y|< 1} \int_{0}^{1}\int_{0}^{r} \left\langle\nabla^{2}V_{\beta}(x+sy),yy^{\top}\right\rangle_{\mathrm{HS}} \,\dif s\,\dif r\,\frac{\dif y}{|y|^{\alpha+d}}\\
    &\qquad\mbox{}+C_{d,\alpha}\int_{|y|\geq 1}\int_{0}^{1} \left\langle\nabla V_{\beta}(x+ry),y\right\rangle \,\dif r\,\frac{\dif y}{|y|^{\alpha+d}}\\
    &\leq C_{d,\alpha}\frac{\beta(3-\beta)}{2}\sqrt{d}\int_{|y|<1}\frac{|y|^{2}}{|y|^{\alpha+d}}\,\dif y
        + C_{d,\alpha}\beta\int_{|y|\geq1}\frac{|x|^{\beta-1}|y|+|y|^{\beta}}{|y|^{\alpha+d}}\,\dif y\\
    &=\frac{C_{d,\alpha}\beta(3-\beta)\sqrt{d}\sigma_{d-1}}{2(2-\alpha)} + C_{d,\alpha}\beta\sigma_{d-1}\left(\frac{|x|^{\beta-1}}{\alpha-1}+\frac{1}{\alpha-\beta}\right).
\end{aligned}
\end{equation}
Again by Young's inequality we get
\begin{gather*}
    \left|(-\Delta)^{\alpha/2}V_{\beta}(x)\right|
    \leq \frac{C_{d,\alpha}\beta(3-\beta)\sqrt{d}\sigma_{d-1}}{2(2-\alpha)}
    + \frac{C_{d,\alpha}\beta\sigma_{d-1}}{\alpha-\beta}+\left(\frac{\theta_{1}}{4}\right)^{1-\beta} \left(\frac{C_{d,\alpha}\sigma_{d-1}}{\alpha-1}\right)^{\beta}
    + \frac{\theta_{1}}{4}V_{\beta}(x).
\end{gather*}
Hence, we have
\begin{align}\label{Lyapunov}
    \mathcal{L}^{\alpha}V_{\beta}(x)
    \leq -\lambda_{1}V_{\beta}(x)+q_{1}\II_{A_{1}}(x),
\end{align}
with $\lambda_{1}=\frac{1}{2} \theta_{1}$,
\begin{gather*}
    q_{1}
    = \theta_{1}\beta+\beta K + \theta_{1}^{1-\beta}|b(0)|^{\beta} + \frac{C_{d,\alpha}\beta(3-\beta)\sqrt{d}\sigma_{d-1}}{2(2-\alpha)}
    + \frac{C_{d,\alpha}\beta\sigma_{d-1}}{\alpha-\beta} + \left(\frac{\theta_{1}}{4}\right)^{1-\beta}\left(\frac{C_{d,\alpha}\sigma_{d-1}}{\alpha-1}\right)^{\beta},
\end{gather*}
and the compact set $A_{1}=\left\{x\in\rd \,:\, |x|\leq \left({4}{\theta_{1}}^{-1}q_{1}\right)^{{1}/{\beta}}\right\}$.

Thus, \cite[Theorem 5.1]{M-T3} yields that the process $(X_{t}^{x})_{t\geq 0}$ is ergodic, i.e.\ there exists a unique invariant probability measure $\mu$ such that for all $x\in\rd$ and $t>0$,
\begin{align*}
    \lim_{t\to\infty}\|P_{t}(x,\cdot)-\mu\|_{\mathrm{TV}}=0,
\end{align*}
where $P_{t}(x,dz)$ is the transition function of the process $(X_{t}^{x})_{t\geq 0}$ and $\|\cdot\|_{\mathrm{TV}}$ denotes the total variation norm on the space of signed measures. Furthermore, because of the inequality above and \cite[Theorem 6.1]{M-T3}, we have
\begin{align*}
    \sup_{|f|\leq V_{\beta}} \big|\Ee [f(X_{t}^{x})]-\mu(f)\big|
    \leq c_1V_{\beta}(x)\eup^{-c_2t}
\end{align*}
for suitable constants $c_1,c_2>0$. In addition, by It\^o's formula, the integrability of $X_{t}^{x}$ can be derived directly from the Lyapunov condition \eqref{Lyapunov}.
\end{proof}

\begin{proof}[Proof of Proposition \ref{W1.1}]
For any $x,y\in\rd$, \eqref{gra} implies that
\begin{align*}
    \left\langle b(x)-b(y),x-y\right\rangle\leq\theta_{2}|x-y|^{2},
\end{align*}
and \eqref{diss} shows for all $|x-y|^2 >{2K}/{\theta_{1}}$ that
\begin{align*}
    \left\langle b(x)-b(y),x-y\right\rangle
    \leq -\theta_{1}|x-y|^{2} + K
    \leq -\frac{\theta_{1}}{2}|x-y|^{2}.
\end{align*}
Hence, we can use \cite[Theorem 1.2]{Wan16} with $K_{1}=\theta_{2}$, $K_{2}=\frac{\theta_{1}}{2}$ and $L_{0}=\frac{2K}{\theta_{1}}$ to get the desired estimate.
\end{proof}

\begin{proof}[Proof of Proposition \ref{W1.1-1}]
We show only \eqref{exact22}, as \eqref{exact2} can be proved in the same way.

Denote by $P(x,\dif y) = \mathds{P}(Y_{1}\in \dif y \mid Y_{0}=x)$. Since $V_{1}(y)\leq|y|+1$ and
\begin{align*}
    Y_{1}
    = x+\eta b(x)+Z_{\eta},
\end{align*}
we have
\begin{align*}
    \int_{\rd}V_{1}(y)\,P(x,\dif y)
    &\leq\int_{\rd}(|y|+1)p\left(\eta,y-x-\eta b(x)\right)\dif y\\
    &=\int_{\rd} \left(|z+x+\eta b(x)|+1\right)p(\eta,z)\,\dif z\\
    &\leq \int_{\rd} \left(|z|+|x+\eta b(x)|+1\right) p(\eta,z)\,\dif z\\
    &\leq \Ee|Z_{\eta}|+\big|x+\eta\big(b(x)-b(0)\big)\big|+\eta|b(0)|+1.
\end{align*}
By \eqref{diss} and \eqref{linear}, we further have
\begin{align*}
    \big|x+\eta\big(b(x)-b(0)\big)\big|^{2}
    &= |x|^{2}+2\eta \left\langle b(x)-b(0),x\right\rangle + \eta^{2}|b(x)-b(0)|^{2}\\
    &\leq \left(1-2\theta_{1}\eta+\theta_{2}^{2}\eta^{2}\right) |x|^{2}+ 2 K\eta,
\end{align*}
which implies
\begin{align*}
    \int_{\rd}V_{1}(y)\,P(x,\dif y)
    &\leq \eta^{\frac{1}{\alpha}}\Ee|Z_{1}|+(1-2\theta_{1} \eta+\theta_{2}^{2}\eta^{2})^{1/2}|x| + \sqrt{2K\eta}+\eta|b(0)|+1\\
    &\leq \left(1-\theta_{1}\eta\right)|x| +\eta^{\frac{1}{\alpha}}\Ee|Z_{1}| + \sqrt{2K\eta} + \eta|b(0)| + 1,
\end{align*}
where the last two inequalities hold because of $\eta < \min\left\{1,\,\theta_{1}\theta_{2}^{-2},\,\theta_{1}^{-1}\right\}$. Hence, we have
\begin{align*}
    \int_{\rd} V_{1}(y)\, P(x,\dif y)
    \leq\lambda_2 V_{1}(x)+q_{2}\II_{A_{2}}(x)
\end{align*}
with
\begin{gather*}
    \lambda_{2}=1-\frac{\theta_{1}}{2}\eta<1,\quad q_{2}=1+\frac{\theta_{1}}{2}\eta
    +\eta^{\frac{1}{\alpha}}\Ee|Z_{1}|+\sqrt{2K\eta}+\eta|b(0)|,\\
    \text{and the compact set\ \ }
    A_{2} = \left\{x\in\rd\,:\, |x|\leq1+\frac{2}{\theta_{1}}\Ee|Z_{1}|\eta^{\frac{1}{\alpha}} + \frac{2\sqrt{2K}}{\theta_{1}}\,\eta^{-\frac{1}{2}}+\frac{2|b(0)|}{\theta_{1}}\right\}.
\end{gather*}
The proof of irreducibility is standard, see e.g.\ \cite[Appendix A]{LTX20}.

We can now use this and \cite[Theorem 6.3]{M-T2} to see that the process $(Y_{k}^{x})_{k\geq 0}$ is exponentially ergodic, i.e.\ there exists a unique invariant probability $\mu_{\eta}$ such that for all $x\in\rd$ and $t>0$,
\begin{gather*}
    \sup_{|f|\leq V_{1}}|\Ee f(Y_{k})-\mu_{\eta}(f)|
    \leq c_1 V_{1}(x)\eup^{-c_2k},
\end{gather*}
for suitable constants $c_1,c_2>0$.
\end{proof}

\begin{proof}[Proof of Lemma \ref{E-Mmoment}]
We show only \eqref{momentd2} as the inequality \eqref{momentd} can be proved in the same way.

Notice that $|y|^{\beta}\leq V_{\beta}(y)$ and
\begin{align*}
    V_{\beta}(Y_{k+1})
    &=V_{\beta}\left(Y_{k} + \eta b(Y_{k}) + Z_{\eta}\right)\\
    &=V_{\beta}\left(Y_{k}+\eta b(Y_{k})\right)+V_{\beta}\left(Y_{k} + \eta b(Y_{k}) + Z_{\eta}\right)-V_{\beta}\left(Y_{k}+\eta b(Y_{k})\right)\\
    &=V_{\beta}(Y_{k})+\int_{0}^{\eta}\left\langle\nabla V_{\beta}\left(Y_{k}+rb(Y_{k})\right),b(Y_{k})\right\rangle\dif r\\
    &\qquad\mbox{}+V_{\beta}\left(Y_{k} + \eta b(Y_{k}) + Z_{\eta}\right)-V_{\beta}\left(Y_{k}+\eta b(Y_{k})\right),
\end{align*}
where $Z_{\eta}$ is independent of $Y_{k}$. Since $\nabla V_{\beta}(x) = {\beta x}{(1+|x|^{2})^{-\frac{2-\beta}{2}}}$,
\eqref{diss} implies that
\begin{equation}\label{modiss1}
\begin{aligned}
    &\int_{0}^{\eta}\left\langle\nabla V_{\beta}\left(Y_{k}+rb(Y_{k})\right),b(Y_{k})\right\rangle\,\dif r\\
    &\qquad\leq \int_{0}^{\eta}\frac{\beta\left\langle Y_{k},b(Y_{k})\right\rangle+\beta r|b(Y_{k})|^{2}}
    {\left(1+|Y_{k}+rb(Y_{k})|^{2}\right)^{\frac{2-\beta}{2}}}\,\dif r\\
    &\qquad\leq\int_{0}^{\eta}\frac{\beta\left\langle Y_{k},b(Y_{k}) -b(0)\right\rangle+\beta|b(0)||Y_{k}|+\beta r|b(Y_{k})|^{2}}
    {\left(1+|Y_{k}+rb(Y_{k})|^{2}\right)^{\frac{2-\beta}{2}}}\,\dif r\\
    &\qquad\leq \int_{0}^{\eta}\frac{-\theta_{1}\beta|Y_{k}|^{2}+\beta K+\beta|b(0)||Y_{k}|+\beta r|b(Y_{k})|^{2}}
    {\left(1+|Y_{k}+rb(Y_{k})|^{2}\right)^{\frac{2-\beta}{2}}}\,\dif r.
\end{aligned}
\end{equation}
One can write by \eqref{linear} and the fact $r\leq \eta\leq\min\left\{1,\frac{\theta_{1}}{8\theta_{2}^{2}},\frac{1}{\theta_{1}}\right\}$
\begin{align*}
    &-\theta_{1}\beta|Y_{k}|^{2}+\beta K+\beta|b(0)||Y_{k}|+\beta r|b(Y_{k})|^{2}\\
    &\qquad\qquad\leq -\frac{3}{4}\theta_{1}\beta|Y_{k}|^{2}+2\eta\theta_{2}^{2}\beta|Y_{k}|^{2}+\frac{\beta|b(0)|^{2}}{\theta_{1}}+2\beta r|b(0)|^{2}+\beta K\\
    &\qquad\qquad\leq -\frac{\theta_{1}\beta}{2}|Y_{k}|^{2}+\frac{\beta|b(0)|^{2}}{\theta_{1}}+2\beta r|b(0)|^{2}+\beta K,
\end{align*}
whereas
\begin{align*}
|Y_{k}&+rb(Y_{k})|^{2}+1\\
    &=|Y_{k}|^{2}+2r\left\langle Y_{k},b(Y_{k})\right\rangle+r^{2}b(Y_{k})^{2}+1\\
    &\leq|Y_{k}|^{2}+2r\left\langle Y_{k},b(Y_{k})-b(0)\right\rangle+2r|b(0)||Y_{k}|+2r^{2}\theta_{2}^{2}|Y_{k}|^{2}+2r^{2}|b(0)|^{2}+1\\
    &\leq (1-2r\theta_{1})|Y_{k}|^{2}+2r|b(0)||Y_{k}|+2r^{2}\theta_{2}^{2}|Y_{k}|^{2}+2\eta^{2}|b(0)|^{2}+1+2\eta K\\
    &\leq|Y_{k}|^{2}+\frac{2r|b(0)|^{2}}{\theta_{1}}+2\eta^{2}|b(0)|^{2}+1+2\eta K.
\end{align*}
Noting that $\left(1+|Y_{k}+rb(Y_{k})|^{2}\right)^{\frac{2-\beta}{2}}\geq1$, we get
\begin{align*}
    &\frac{-\theta_{1}\beta|Y_{k}|^{2}+\beta K+\beta|b(0)||Y_{k}|+\beta r|b(Y_{k})|^{2}}
    {\left(1+|Y_{k}+rb(Y_{k})|^{2}\right)^{\frac{2-\beta}{2}}}\\
    &\quad\leq -\frac{\theta_{1}\beta}{2}\,\frac{|Y_{k}|^{2}}{\left(|Y_{k}|^{2}+\frac{2r|b(0)|^{2}}{\theta_{1}}+2\eta^{2}|b(0)|^{2}+1+2\eta K\right)^{\frac{2-\beta}{2}}} + \frac{\beta|b(0)|^{2}}{\theta_{1}}+2\beta r|b(0)|^{2}+\beta K\\
    &\quad\leq -\frac{\theta_{1}\beta}{2}\,\left(|Y_{k}|^{2}+\frac{2r|b(0)|^{2}}{\theta_{1}}+2\eta^{2}|b(0)|^{2}+1+2\eta K\right)^{\frac{\beta}{2}}+ C_1\\
    &\quad\leq -\frac{\theta_{1}\beta}{2}\,V_{\beta}(Y_k)+C_1,
\end{align*}
where
\begin{gather*}
    C_1:=\frac{\theta_{1}\beta}{2} \left(\frac{2r|b(0)|^{2}}{\theta_{1}}+2\eta^{2}|b(0)|^{2}+1+2\eta K\right)+ \frac{\beta|b(0)|^{2}}{\theta_{1}}+2\beta r|b(0)|^{2}+\beta K.
\end{gather*}
Combining this with \eqref{modiss1}, we arrive at
\begin{gather*}
    \int_{0}^{\eta}\left\langle\nabla V_{\beta}\left(Y_{k}+rb(Y_{k})\right),b(Y_{k})\right\rangle \dif r
    \leq -\frac{\theta_{1}\beta}{2}\,\eta V_{\beta}(Y_{k})+C_1\eta.
\end{gather*}
In addition, for any $y\in\rd$, It\^{o}'s formula and the inequality \eqref{fracbound} imply that
\begin{align*}
    &\left|\Ee \left[V_{\beta}\left(y+ \eta b(y) + Z_{\eta}\right)-V_{\beta}\left(y+\eta b(y)\right)\right]\right|\\
    &=\left|\int_{0}^{\eta}\Ee\left[(-\Delta)^{\alpha/2}V_{\beta}\left(y+ \eta b(y) + Z_{r}\right)\right]\dif r\right|\\
    &\leq\int_{0}^{\eta}\left[\frac{C_{d,\alpha}\beta(3-\beta)\sqrt{d}\sigma_{d-1}}{2(2-\alpha)} + C_{d,\alpha}\beta\sigma_{d-1}\left(\frac{\Ee|y+ \eta b(y) + Z_{r}|^{\beta-1}}{\alpha-1}+\frac{1}{\alpha-\beta}\right)\right]\dif r\\
    &\leq C_{d,\alpha}\beta\sigma_{d-1}\int_{0}^{\eta}\left[\frac{(3-\beta)\sqrt{d}}{2(2-\alpha)} + \frac{|y|^{\beta-1}+ \eta^{\beta-1}|b(y)|^{\beta-1} + \Ee|Z_{r}|^{\beta-1}}{\alpha-1}+\frac{1}{\alpha-\beta}\right]\dif r.
\end{align*}
This yields, in turn,
\begin{align*}
    &\left|\Ee \left[V_{\beta}\left(y+ \eta b(y) + Z_{\eta}\right)-V_{\beta}\left(y+\eta b(y)\right)\right]\right|\\
    &\leq C_{d,\alpha}\beta\sigma_{d-1}\left(\frac{(3-\beta)\sqrt{d}\eta}{2(2-\alpha)}+\frac{\eta}{\alpha-\beta}+\frac{1+\theta_{2}^{\beta-1}}{\alpha-1}\eta|y|^{\beta-1}+\eta|b(0)|^{\beta-1}
    +\int_{0}^{\eta}\frac{\Ee|Z_{r}|^{\beta-1}}{\alpha-1}\dif r\right)\\
    &\leq C_{d,\alpha}\beta\sigma_{d-1}\left(\frac{(3-\beta)\sqrt{d}\eta}{2(2-\alpha)}+\frac{\eta}{\alpha-\beta}+\frac{1+\theta_{2}^{\beta-1}}{\alpha-1}\eta|y|^{\beta-1}+\eta|b(0)|^{\beta-1}
    +\frac{\Ee|Z_{1}|^{\beta-1}}{\alpha-1}\eta\right),
\end{align*}
where the first inequality uses \eqref{linear} and the fact that $0<\eta<1$.
Since $Z_{\eta}$ is independent of $Y_{k}$, we can derive that
\begin{align*}
    &\left|V_{\beta}\left(Y_{k} + \eta b(Y_{k}) + Z_{\eta}\right)-V_{\beta}\left(Y_{k}+\eta b(Y_{k})\right)\right|\\
    &\leq C_{d,\alpha}\beta\sigma_{d-1}\left(\frac{(3-\beta)\sqrt{d}\eta}{2(2-\alpha)}+\frac{\eta}{\alpha-\beta}+\frac{1+\theta_{2}^{\beta-1}}{\alpha-1}\eta\Ee|Y_{k}|^{\beta-1}+\eta|b(0)|^{\beta-1}
    +\frac{\Ee|Z_{1}|^{\beta-1}}{\alpha-1}\eta\right)\\
    &\leq\frac{\theta_{1}(\beta-1)}{2}\eta\,V_{\beta}(Y_k)+C_2\eta,
\end{align*}
where the last inequality follows from Young's inequality and
\begin{align*}
    C_{2}
    &=C_{d,\alpha}\beta\sigma_{d-1}\left(\frac{(3-\beta)\sqrt{d}}{2(2-\alpha)}+\frac{1}{\alpha-\beta}+|b(0)|^{\beta-1}
    +\frac{\Ee|Z_{1}|^{\beta-1}}{\alpha-1}\right)\\
    &\qquad+\left(\frac{C_{d,\alpha}\sigma_{d-1}(1+\theta_{2}^{\beta-1})}{\alpha-1}\right)^{\beta}\left(\frac{2}{\theta_{1}}\right)^{\beta-1}.
\end{align*}
Therefore,
\begin{align*}
    \Ee [V_{\beta}(Y_{k+1})]
    \leq \left(1-\frac{\theta_{1}}{2}\eta\right)\Ee [V_{\beta}(Y_{k})] + (C_1+C_2)\eta,
\end{align*}
which we can iterate this to get
\begin{align*}
    \Ee [V_{\beta}(Y_{k+1})]
    &\leq \left(1-\frac{\theta_{1}}{2}\eta\right)^{k+1}V_{\beta}(x) + (C_1+C_2)\eta\sum_{j=0}^{k} \left(1-\frac{\theta_{1}}{2}\eta\right)^{j}\\
    &\leq V_{\beta}(x)+\frac{2(C_1+C_2)}{\theta_{1}}.
\end{align*}
Using that $V_{\beta}(x)\leq 1+|x|^{\beta}$, we finally get
\begin{align*}
    \Ee |Y_{k}^{x}|^{\beta}
    \leq \Ee [V_{\beta}(Y_{k}^{x})]\leq C(1+|x|^{\beta}),
\end{align*}
for some constant $C$ which is independent of $\eta$.
\end{proof}

\section{Exact rate for the Ornstein--Uhlenbeck process}\label{OU}

In this section, we assume that $\mu$ is the invariant measure of the Ornstein--Uhlenbeck process on $\real$:
\begin{align}\label{OUSDE}
    \dif X_{t}=-X_{t}\,\dif t+\dif Z_{t}, \quad X_{0}=x,
\end{align}
where $Z_{t}$ is a rotationally symmetric $\alpha$-stable L\'evy process ($1<\alpha<2$), and $\tilde{\mu}_{\eta}$ is the invariant measure of
\begin{align*}
    \tilde{Y}_{k+1} = \tilde{Y}_{k}-\eta \tilde{Y}_{k}+\frac{\eta^{1/\alpha}}{\sigma}\widetilde{Z}_{k+1}, \quad k=0,1,2,\dots,
\end{align*}
where $\eta\in(0,1)$, $\tilde{Y}_{0}=x$, $\sigma=\big(\frac{\alpha}{2d_{\alpha}}\big)^{1/\alpha}$ with $d_{\alpha}=C_{1,\alpha}=\left(2\int_{0}^{\infty}\frac{1-\cos y}{y^{\alpha+1}}\,\dif y\right)^{-1}$, and $\widetilde{Z}_{j}$ are i.i.d.\ random variables with density
\begin{align}\label{symmetric}
    p(z) = \frac{\alpha}{2|z|^{\alpha+1}}\,\II_{(1,\infty)}(|z|).
\end{align}

\begin{proposition}\label{ourate}
    For every $x\in\real$ and $\alpha\in(1,2)$,
    \begin{gather*}
        0<
        \liminf_{\eta\downarrow 0} \frac{W_1(\mu,\tilde{\mu}_\eta)}{\eta^{2/\alpha-1}}
        \leq\limsup_{\eta\downarrow 0} \frac{W_1(\mu,\tilde{\mu}_\eta)}{\eta^{2/\alpha-1}}
        <\infty.
    \end{gather*}
\end{proposition}

\begin{proof}
Since $X_{t}=\eup^{-t}x+\eup^{-t}\int_{0}^{t}\eup^{s}\,\dif Z_{s}$, we get
\begin{align*}
    \Ee \left[\eup^{\iup\xi X_{t}}\right]
    = \eup^{\iup\xi\eup^{-t}x} \Ee \left[ \eup^{\iup\int_{0}^{t}\xi \eup^{-t}\eup^{s}\,\dif Z_{s}}\right]
    &=\eup^{\iup\xi\eup^{-t}x}\eup^{-\int_{0}^{t}|\xi \eup^{-t}\eup^{s}|^{\alpha}\,\dif s}\\
    &=\eup^{\iup\xi\eup^{-t}x} \eup^{-\alpha^{-1}|\xi|^{\alpha}(1-\eup^{-\alpha t})}\\
    &\xrightarrow[t\to\infty]{} \eup^{-\alpha^{-1}|\xi|^{\alpha}}
    =\Ee \left[\eup^{\iup\xi\alpha^{-1/\alpha}Z_{1}}\right].
\end{align*}
Hence, the invariant measure $\mu$ is given by the law of $\alpha^{-1/\alpha}Z_{1}$.

It is easy to see that
\begin{align*}
    \tilde{Y}_{k+1} = (1-\eta)^{k+1}x + \frac{\eta^{1/\alpha}}{\sigma}\sum_{i=0}^{k}(1-\eta)^{i}\widetilde{Z}_{k+1-i}.
\end{align*}
Denote by $\varphi(\xi) = \Ee \bigl[\eup^{\iup\xi\widetilde{Z}_{j}}\bigr]$ the characteristic function of the Pareto distribution. Then we have
\begin{align*}
    \Ee \left[\eup^{\iup\xi \tilde{Y}_{k+1}}\right]
    = \eup^{\iup\xi(1-\eta)^{k+1}x} \prod_{i=0}^{k}\Ee \left[\eup^{\iup\xi\frac{\eta^{1/\alpha}}{\sigma}(1-\eta)^{i}\widetilde{Z}_{k+1-i}}\right]
    = \eup^{\iup\xi(1-\eta)^{k+1}x} \prod_{i=0}^{k}\varphi\left(\frac{\eta^{1/\alpha}}{\sigma}(1-\eta)^{i}\xi\right).
\end{align*}
Letting $k\to\infty$ and denoting by $\tilde{Y}_\eta$ a random variable with distribution $\tilde{\mu}_\eta$, we get
\begin{equation}\label{limitcharc}
    \Ee \left[\eup^{\iup\xi \tilde{Y}_\eta}\right]
    = \prod_{i=0}^\infty\varphi\left(\frac{\eta^{1/\alpha}}{\sigma}(1-\eta)^{i}\xi\right).
\end{equation}

For $\xi>0$, we have
\begin{align*}
    1-\varphi(\xi)
    &= 2\int_1^\infty \left[1-\cos(\xi z)\right] p(z)\,\dif z
    =\alpha\int_1^\infty \left[1-\cos(\xi z)\right] \frac{\dif z}{z^{\alpha+1}}\\
    &=\alpha\xi^{\alpha}\left(\int_{0}^{\infty}\left[1-\cos u\right] \frac{\dif u}{u^{\alpha+1}}
        - \int_{0}^{\xi} \left[1-\cos u\right] \frac{\dif u}{u^{\alpha+1}}\right)\\
    &=\sigma^\alpha\xi^\alpha - \alpha\xi^\alpha\int_{0}^{\xi} \left[1-\cos u\right] \frac{\dif u}{u^{\alpha+1}}.
\end{align*}
Since $p(z)$ is symmetric, cf.\ \eqref{symmetric}, we have $\varphi(\xi)=\varphi(-\xi)$, and so
\begin{gather*}
    \varphi(\xi)
    = 1-\sigma^\alpha|\xi|^\alpha+\alpha|\xi|^\alpha \int_{0}^{|\xi|}\left[1-\cos u\right] \frac{\dif u}{u^{\alpha+1}}.
\end{gather*}
for all $\xi\in\real$. Since $c:=\inf_{0<u\leq1}\left(1-\cos u\right)/u^2 > 0$, we get for all $|\xi|\leq 1$,
\begin{gather*}
    \varphi(\xi)
    \geq 1-\sigma^\alpha|\xi|^\alpha + \alpha|\xi|^\alpha\int_{0}^{|\xi|} cu^2 \frac{\dif u}{u^{\alpha+1}}
    = 1-\sigma^\alpha|\xi|^\alpha + \frac{c\alpha}{2-\alpha}\,|\xi|^2.
\end{gather*}
Thus, for $|\xi|\leq1$ and $0 < \eta < 1\wedge\sigma^\alpha$,
\begin{equation}\label{lowerb}
\begin{aligned}
    \log\varphi\left(\frac{\eta^{1/\alpha}}{\sigma}(1-\eta)^{i}\xi\right)
    &\geq\log\left(1-\sigma^\alpha\left|\frac{\eta^{1/\alpha}}{\sigma}(1-\eta)^{i}\xi
    \right|^\alpha + \frac{c\alpha}{2-\alpha}\left| \frac{\eta^{1/\alpha}}{\sigma}(1-\eta)^{i}\xi\right|^2\right)\\
    &=\log\left(1-\eta(1-\eta)^{\alpha i}|\xi|^\alpha + \frac{c\alpha}{(2-\alpha)\sigma^2}\,\eta^{2/\alpha}(1-\eta)^{2i}|\xi|^2\right).
\end{aligned}
\end{equation}
Observe that
\begin{gather*}
    \lim_{x\downarrow 0} \frac{\log\left(1-x+\frac{c\alpha}{(2-\alpha)\sigma^2}\,x^{2/\alpha}\right)+x}{x^{2/\alpha}}
    = \frac{c\alpha}{(2-\alpha)\sigma^2}.
\end{gather*}
Therefore, there is some constant $C=C(\alpha,\sigma)>0$ such that for small enough $x>0$
\begin{gather*}
    \log\left( 1-x+\frac{c_1\alpha}{(2-\alpha)\sigma^2}\,x^{2/\alpha}\right)
    \geq -x + Cx^{2/\alpha}.
\end{gather*}
If we use this in \eqref{lowerb}, we obtain for all $|\xi|\leq1$ and small enough $\eta>0$,
\begin{gather*}
    \log\varphi\left(\frac{\eta^{1/\alpha}}{\sigma}(1-\eta)^{i}\xi\right)
    \geq -\eta(1-\eta)^{\alpha i}|\xi|^\alpha + C\eta^{2/\alpha}(1-\eta)^{2i}|\xi|^2.
\end{gather*}
Inserting this into \eqref{limitcharc}, we see for all $|\xi|\leq1$ and small enough $\eta>0$
\begin{align*}
    \log\Ee \left[\eup^{\iup\xi \tilde{Y}_\eta}\right]
    &=\sum_{i=0}^\infty\log\varphi\left(\frac{\eta^{1/\alpha}}{\sigma}(1-\eta)^{i}\xi\right)\\
    &\geq-\eta|\xi|^\alpha\sum_{i=0}^\infty(1-\eta)^{\alpha i} + C\eta^{2/\alpha}|\xi|^2 \sum_{i=0}^\infty(1-\eta)^{2i}\\
    &=-|\xi|^\alpha\frac{\eta}{1-(1-\eta)^\alpha} + C|\xi|^2\frac{\eta^{2/\alpha}}{1-(1-\eta)^2}\\
    &=-\frac{1}{\alpha}\,|\xi|^\alpha - |\xi|^\alpha \Omega(\eta)+|\xi|^2 \Omega(\eta^{2/\alpha-1}).
\end{align*}
In the last equality we use that
\begin{gather*}
    \lim_{\eta\downarrow 0} \left[\frac{\eta}{1-(1-\eta)^{\alpha}} - \frac{1}{\alpha}\right] \eta^{-1}
    = \frac{\alpha-1}{2\alpha}
    \quad\text{and}\quad
    \lim_{\eta\downarrow 0} \frac{\eta^{2/\alpha}}{1-(1-\eta)^2}\,\eta^{-2/\alpha+1}
    = \frac 12.
\end{gather*}
Here and in the following, the notation $f(\eta)=\Omega(g(\eta))$ as $\eta\downarrow0$ means
that $\lim_{\eta\downarrow0}\frac{f(\eta)}{g(\eta)}$ is a
positive (finite) constant, where $f$ and $g$ are some positive functions.
With the elementary inequality $\eup^x \geq 1+x$ for $x\in\real$ we see for all $|\xi|\leq1$ and sufficiently small $\eta>0$ that
\begin{align*}
    \Ee \left[\eup^{\iup\xi \tilde{Y}_\eta}\right]
    &\geq \exp\left[ -\frac{1}{\alpha}\,|\xi|^\alpha - |\xi|^\alpha \Omega(\eta)+|\xi|^2 \Omega(\eta^{2/\alpha-1}) \right]\\
    &\geq \eup^{-|\xi|^\alpha/\alpha} \left[1-|\xi|^\alpha \Omega(\eta)+|\xi|^2 \Omega(\eta^{2/\alpha-1}) \right],
\end{align*}
which yields
\begin{equation}\label{wlower}
\begin{aligned}
    \int_{-1}^1\left( \Ee \left[\eup^{\iup\xi \tilde{Y}_\eta}\right] - \Ee \left[\eup^{\iup\xi \alpha^{-1/\alpha}Z_1}\right] \right)\dif\xi
    &\geq \int_{-1}^1 \eup^{-\alpha^{-1}|\xi|^\alpha} \left[-|\xi|^\alpha \Omega(\eta)+|\xi|^2 \Omega(\eta^{2/\alpha-1}) \right]\dif\xi\\
    &= -\Omega(\eta)+ \Omega(\eta^{2/\alpha-1})
    = \Omega(\eta^{2/\alpha-1}).
\end{aligned}
\end{equation}

Define
\begin{gather*}
    h(x)
    :=\frac{1}{M}\left(\frac{\sin x}{x}\II_{\{x\neq0\}} + \II_{\{x=0\}}\right),\quad x\in\real,
\intertext{where}
    M
    :=\sup_{x\in\real\setminus\{0\}}\left|\frac{x\cos x-\sin x}{x^2}\right|\in(0,\infty).
\end{gather*}
Since $h\in\mathrm{Lip}(1)$ and
\begin{gather*}
    h(x)
    =
    \frac{1}{2M}\int_{-1}^1\eup^{\iup\xi x}\,\dif\xi,
    \quad x\in\real,
\end{gather*}
it follows from Fubini's theorem and \eqref{wlower} that
\begin{align*}
    W_1(\mu,\tilde{\mu}_\eta)
    &\geq \left|\Ee \left[h(\tilde{Y}_\eta)\right] - \Ee \left[h(\alpha^{-1/\alpha}Z_1)\right]\right|\\
    &=\left|\int_{\real}\left(\frac{1}{2M}\int_{-1}^1\eup^{\iup\xi x}\,\dif\xi\right)\mathds{P}(\tilde{Y}_\eta\in\dif x)
    -\int_{\real}\left(\frac{1}{2M}\int_{-1}^1\eup^{\iup\xi x}\,\dif\xi\right)\mathds{P}(\alpha^{-1/\alpha}Z_1\in\dif x)\right|\\
    &=\left|\frac{1}{2M}\int_{-1}^1\Ee \left[\eup^{\iup\xi \tilde{Y}_\eta}\right] \dif\xi
    - \frac{1}{2M}\int_{-1}^1\Ee \left[\eup^{\iup\xi \smash{\alpha^{-1/\alpha}}Z_1}\right]\dif\xi \right|\\
    &=\frac{1}{2M}\left|\int_{-1}^1\left(\Ee \left[\eup^{\iup\xi \tilde{Y}_\eta}\right]
    - \Ee \left[\eup^{\iup\xi \smash{\alpha^{-1/\alpha}}Z_1}\right]\right)\dif\xi\right|\\
    &\geq \Omega(\eta^{2/\alpha-1}).
\end{align*}
Combining this with the upper bound in Theorem \ref{main}\,\eqref{main-ii}, finishes the proof.
\end{proof}

{

\section{Alternative proof of Lemma \ref{regular} by Malliavin calculus \cite{Nua06}}\label{supple}

In this section, we will prove Lemma \ref{regular} by Malliavin calculus established in \cite{Nua06}.
Keep the same notations as in Subsection \ref{timec} and we consider the integral form of SDE \eqref{xellepsilon}, that is,
\begin{gather}\label{NSDE}
    X_{t}^{x;l^\epsilon}
    = x+\int_{0}^{t}b(X_{r}^{x;l^\epsilon})\,\dif r
    +W_{l^\epsilon_{t}-l^\epsilon_{0}}.
\end{gather}
Then $X_{t}^{x;l^\epsilon}$ is Malliavin differentiable. In the following we will use the standard notations from \cite{Nua06}.

Recall that $\gamma^{\epsilon}$ is the inverse function of $l^{\epsilon}$. The Malliavin derivative
$D_{s}X_{t}^{x;l^\epsilon}=0$ for $s>l_{t}^{\epsilon}$, which amounts to $\gamma^{\epsilon}_{s}>t$. Thus,
\begin{align*}
D_{s}X_{t}^{x;l^\epsilon}=\II_{\{t\geq\gamma^{\epsilon}_{s}\}}D_{s}X_{t}^{x;l^\epsilon}.
\end{align*}
For $t\geq\gamma^{\epsilon}_{s}$, \eqref{NSDE} implies
\begin{align*}
D_{s}X_{t}^{x;l^\epsilon}=I_{d}+\int_{\gamma^{\epsilon}_{s}}^{t}\nabla b(X_{r}^{x;l^\epsilon})D_{s}X_{r}^{x;l^\epsilon}\,\dif r.
\end{align*}
Now, we consider the tangent flow $\nabla_{x} X_{t}^{x;l^\epsilon}$, which verifies the equation
\begin{align*}
\nabla_{x} X_{t}^{x;l^\epsilon}&=I_{d}+\int_{0}^{t}\nabla b(X_{r}^{x;l^\epsilon})\nabla_{x}X_{r}^{x;l^\epsilon}\,\dif r\\
&=\nabla_{x}X_{\gamma^{\epsilon}_{s}}^{x;l^\epsilon}+\int_{\gamma^{\epsilon}_{s}}^{t}\nabla b(X_{r}^{x;l^\epsilon})\nabla_{x}X_{r}^{x;l^\epsilon}\,\dif r.
\end{align*}
It is clear that $\nabla_{x}X_{\gamma^{\epsilon}_{s}}^{x;l^\epsilon}$ is invertible, so it holds that
\begin{align*}
D_{s}X_{t}^{x;l^\epsilon}=\nabla_{x} X_{t}^{x;l^\epsilon}(\nabla_{x}X_{\gamma^{\epsilon}_{s}}^{x;l^\epsilon})^{-1}.
\end{align*}
Using the chain rule, for $\phi\in \mathcal{C}_{b}^{1}(\rd,\real)$,
\begin{align*}
D_{s}\phi(X_{t}^{x;l^\epsilon})=\nabla\phi(X_{t}^{x;l^\epsilon})D_{s}X_{t}^{x;l^\epsilon}=\nabla\phi(X_{t}^{x;l^\epsilon})\nabla_{x} X_{t}^{x;l^\epsilon}(\nabla_{x}X_{\gamma^{\epsilon}_{s}}^{x;l^\epsilon})^{-1}.
\end{align*}
Hence, we have
\begin{align*}
\nabla\phi(X_{t}^{x;l^\epsilon})\nabla_{x} X_{t}^{x;l^\epsilon}=D_{s}\phi(X_{t}^{x;l^\epsilon})\nabla_{x}X_{\gamma^{\epsilon}_{s}}^{x;l^\epsilon}.
\end{align*}

Based on the above framework, one can derive the following Bismut--Elworthy--Li formula.
\begin{lemma}\label{B-E-L}
Let $G$ be a $d$-dimensional vector which is Malliavin differentiable. Then for any $f\in \mathcal{C}_{b}^{1}(\rd,\real)$,
\begin{align*}
\mathbb{E}[\nabla f(X_{t}^{x;l^\epsilon})\nabla_{x} X_{t}^{x;l^\epsilon}G]=\frac{1}{l^\epsilon_{t}-l^\epsilon_{0}}\,\mathbb{E}\left[f(X_{t}^{x;l^\epsilon})
\left(G\int_{l^\epsilon_{0}}^{l^\epsilon_{t}}\nabla_{x}X_{\gamma^{\epsilon}_{s}}^{x;l^\epsilon}\,\dif W_{s-l^\epsilon_{0}}+\int_{l^\epsilon_{0}}^{l^\epsilon_{t}}\nabla_{x}X_{\gamma^{\epsilon}_{s}}^{x;l^\epsilon}D_{s}G\,\dif s\right)\right].
\end{align*}
\end{lemma}
\begin{proof}
For any $s\in[l^\epsilon_{0},l^\epsilon_{t}]$ (this amounts to $t\geq\gamma^{\epsilon}_{s}$), we have
\begin{align*}
\mathbb{E}\left[\nabla f(X_{t}^{x;l^\epsilon})\nabla_{x} X_{t}^{x;l^\epsilon}G\right]=\mathbb{E}\left[D_{s}f(X_{t}^{x;l^\epsilon})\nabla_{x}X_{\gamma^{\epsilon}_{s}}^{x;l^\epsilon}G\right].
\end{align*}
Integrating over $s\in[l^\epsilon_{0},l^\epsilon_{t}]$, one can derive from the duality relation \cite[(1.42)]{Nua06} that
\begin{align*}
\mathbb{E}\left[\nabla f(X_{t}^{x;l^\epsilon})\nabla_{x} X_{t}^{x;l^\epsilon}G\right]&=\frac{1}{l^\epsilon_{t}-l^\epsilon_{0}}\int_{l^\epsilon_{0}}^{l^\epsilon_{t}}\mathbb{E}
\left[D_{s}f(X_{s}^{x;l^\epsilon})\nabla_{x}X_{\gamma^{\epsilon}_{s}}^{x;l^\epsilon}G\right]\dif s\\
&=\frac{1}{l^\epsilon_{t}-l^\epsilon_{0}}\,\mathbb{E}\left[f(X_{t}^{x;l^\epsilon})
\int_{l^\epsilon_{0}}^{l^\epsilon_{t}}\nabla_{x}X_{\gamma^{\epsilon}_{s}}^{x;l^\epsilon}G\,\tilde{\dif}W_{s-l^\epsilon_{0}} \right],
\end{align*}
where $\tilde{\dif}W_{s-l^\epsilon_{0}}$ denotes the Skorohod integral (see \cite[Subsection 1.3.2]{Nua06}).
Recall the following well-known formula (cf.\ \cite[(1.56)]{Nua06})
\begin{align*}
\int_{l^\epsilon_{0}}^{l^\epsilon_{t}}\nabla_{x}X_{\gamma^{\epsilon}_{s}}^{x;l^\epsilon}G\,\tilde{\dif}W_{s-l^\epsilon_{0}}
=G\int_{l^\epsilon_{0}}^{l^\epsilon_{t}}\nabla_{x}X_{\gamma^{\epsilon}_{s}}^{x;l^\epsilon}\,\tilde{\dif} W_{s-l^\epsilon_{0}}+\int_{l^\epsilon_{0}}^{l^\epsilon_{t}}\nabla_{x}X_{\gamma^{\epsilon}_{s}}^{x;l^\epsilon}D_{s}G\,\dif s.
\end{align*}
Noting that $X_{\gamma^{\epsilon}_{s}}^{x;l^\epsilon}$ depends on $W_{r}$ with $r\leq s-l^\epsilon_{0}$, it follows that $s\mapsto\nabla_{x}X_{\gamma^{\epsilon}_{s}}^{x;l^\epsilon}$ is $\mathcal{F}_{s-l^\epsilon_{0}}^{\mathds{W}}$-adapted and consequently the Skorohod integral coincides with the It\^{o} integral (see \cite[Subsection 1.3.3]{Nua06}). Then the desired result follows.
\end{proof}

Now, we can use the above Bismut-Elworthy-Li formula to prove the following lemma.
\begin{lemma}\label{supplelem}
    Under \textbf{Assumption A}, for all $\phi\in \mathcal{C}_{b}^{2}(\rd,\real)$, $x\in\rd$ and $t\in(0,1]$,
    \begin{align*}
    \left\|\Ee \left[\nabla^{2}_x
    \phi\bigl(X_{t}^{x;l}\bigr)(x)\right]\right\|_{{\rm HS}}\leq C\|\nabla\phi\|_{\infty}\left(
    1+\frac{1}{ \sqrt{l_{t}}}
    \right).
    \end{align*}
\end{lemma}
\begin{proof}
By the chain rule, for $i,j=1,\dots,d$,
$$
\partial_{j}\partial_{i}
    \phi\bigl(X_{t}^{x;l^{\epsilon}}\bigr)
    =\sum_{k=1}^{d}\partial_{k}\phi\bigl(X_{t}^{x;l^{\epsilon}}\bigr)\partial_{j}\partial_{i}X_{t;k}^{x;l^{\epsilon}}+
    \sum_{p=1}^{d}\sum_{k=1}^{d}\partial_{p}\partial_{k}
    \phi\bigl(X_{t}^{x;l^{\epsilon}}\bigr)\partial_{j}X_{t;p}^{x;l^{\epsilon}}\partial_{i}
    X_{t;k}^{x;l^{\epsilon}},
$$
where $X_{t;k}^{x;l^{\epsilon}}$ is the $k$-th component of the $d$-dimensional vector $X_{t}^{x;l^{\epsilon}}$, $k=1,2,\dots,d$.
Using Lemma \ref{B-E-L} with $d=1$, $G=\partial_{i}X_{t;k}^{x;l^{\epsilon}}$ and $f=\partial_{k}\phi$, we have
\begin{align*}
&\mathbb{E}\left[\partial_{p}\partial_{k}\phi\bigl(X_{t}^{x;l^{\epsilon}}\bigr)\partial_{j}X_{t;p}^{x;l^{\epsilon}}\partial_{i}X_{t;k}^{x;l^{\epsilon}}\right]\\
&\qquad=\frac{1}{l^\epsilon_{t}-l^\epsilon_{0}}\,\mathbb{E}\left[\partial_{k}\phi(X_{t}^{x;l^\epsilon})
\left(\partial_{i}X_{t;k}^{x;l^{\epsilon}}\int_{l^\epsilon_{0}}^{l^\epsilon_{t}}\partial_{j}X_{\gamma^{\epsilon}_{s};p}^{x;l^\epsilon}
\,\dif W_{s-l^\epsilon_{0};p}
+\int_{l^\epsilon_{0}}^{l^\epsilon_{t}}\partial_{j}X_{\gamma^{\epsilon}_{s};p}^{x;l^\epsilon}D_{s}\partial_{i}X_{t;k}^{x;l^{\epsilon}}\,\dif s\right)\right],
\end{align*}
where $W_{s-l^\epsilon_{0};p}$ is the $p$-th component of the $d$-dimensional Brownian motion $W_{s-l^\epsilon_{0}}$, $p=1,\dots,d$.
Then we obtain that
\begin{align*}
&\left|\Ee \left[\partial_{j}\partial_{i}
    \phi\bigl(X_{t}^{x;l^{\epsilon}}\bigr)\right]\right|
\leq\|\nabla\phi\|_{\infty}\sum_{k=1}^{d}\mathbb{E}\left|\partial_{j}\partial_{i}X_{t;k}^{x;l^{\epsilon}}\right|\\
&\qquad\qquad+
\frac{\|\nabla\phi\|_{\infty}}{l^\epsilon_{t}-l^\epsilon_{0}}
\sum_{p=1}^{d}\sum_{k=1}^{d}\mathbb{E}\left[\left|\partial_{i}X_{t;k}^{x;l^{\epsilon}}
\int_{l^\epsilon_{0}}^{l^\epsilon_{t}}\partial_{j}X_{\gamma^{\epsilon}_{s};p}^{x;l^\epsilon}\,\dif W_{s-l^\epsilon_{0};p}\right|+\left|\int_{l^\epsilon_{0}}^{l^\epsilon_{t}}\partial_{j}
X_{\gamma^{\epsilon}_{s};p}^{x;l^\epsilon}D_{s}\partial_{i}X_{t;k}^{x;l^{\epsilon}}\,\dif s\right|\right].
\end{align*}
In addition, by a similar argument as in the proof of Lemma \ref{suppleSLS}, it is not hard to verify that for
any $t\in(0,1]$, $s\in[l^\epsilon_{0},l^\epsilon_{t}]$ and $i,j,k=1,\dots,d$, the derivatives
$\partial_{j}\partial_{i}X_{t;k}^{x;l^{\epsilon}}$,
$\partial_{i}X_{t;k}^{x;l^{\epsilon}}$ and $D_{s}\partial_{i}X_{t;k}^{x;l^{\epsilon}}$ are all bounded.
Then it follows from the Cauchy--Schwarz inequality and It\^{o}'s isometry that
\begin{align*}
\left|\Ee \left[\partial_{j}\partial_{i}
    \phi\bigl(X_{t}^{x;l^{\epsilon}}\bigr)\right]\right|
    &\leq C\|\nabla\phi\|_{\infty}+\frac{C\|\nabla\phi\|_{\infty}}{l^\epsilon_{t}-l^\epsilon_{0}}
    \left(
    \sqrt{l^\epsilon_{t}-l^\epsilon_{0}}+l^\epsilon_{t}-l^\epsilon_{0}
    \right)\\
    &\leq C\|\nabla\phi\|_{\infty}\left(
    1+\frac{1}{ \sqrt{l^\epsilon_{t}-l^\epsilon_{0}}}
    \right),
\end{align*}
which further implies
\begin{align*}
\left\|\Ee \left[\nabla^{2}_{x}
    \phi\bigl(X_{t}^{x;l^{\epsilon}}\bigr)\right]\right\|_{{\rm HS}}\leq C\|\nabla\phi\|_{\infty}\left(
    1+\frac{1}{ \sqrt{l^\epsilon_{t}-l^\epsilon_{0}}}
    \right).
\end{align*}
By the same argument as in the proof of \cite[Lemma 2.2]{Z}, one can prove that
\begin{equation}
\begin{aligned}
    \lim_{\epsilon\downarrow0}\Ee \left[\nabla^{2}_{x}
    \phi\bigl(X_{t}^{x;l^{\epsilon}}\bigr)\right]
    = \Ee \left[\nabla^{2}_{x}
    \phi\bigl(X_{t}^{x;l}\bigr)\right].
\end{aligned}
\end{equation}
It remains to let $\epsilon\downarrow0$ to get the desired estimate.
\end{proof}

With the help of Lemma \ref{supplelem}, we can prove Lemma \ref{regular} as in the proof
presented in Subsection \ref{suppleimp}.
}
\end{appendix}

\noindent
\textbf{Acknowledgement.} We would like to gratefully thank the associate editor and the two anonymous referees for their very
helpful and constructive comments. The research of L.\ Xu is supported in part by NSFC No.\ 12071499,  Macao S.A.R.\ grant FDCT  0090/2019/A2 and University of Macau grant  MYRG2020-00039-FST. R.L.\ Schilling was supported through the joint Polish--German NCN--DFG `Beethoven 3' grant (NCN 2018/31/G/ST1/02252; DFG SCHI 419/11-1). C.-S.\ Deng is supported by
Natural Science Foundation of Hubei Province of China (2022CFB129).
P.\ Chen is supported by the NSF of Jiangsu Province grant BK20220867 and the Initial Scientific Research Fund of Young Teachers in Nanjing University of Aeronautics and Astronautics (1008-YAH21111).

\bibliographystyle{amsplain}

\begin{thebibliography}{99}\frenchspacing

\bibitem{BT96}
V.\ Bally and D.\ Talay (1996):
The law of the Euler scheme for stochastic differential equations.
\textit{Probability Theory and Related Fields}  \textbf{104}, 43--60.

\bibitem{BF75}
C.\ Berg and G.\ Forst (1975):
\textit{Potential Theory on Locally Compact Abelian Groups}.
Springer, Ergebnisse der Mathematik und ihrer Grenzgebiete. II.~Ser.\ Bd.~\textbf{87}, Berlin.

\bibitem{BG60}
R.M.\ Blumenthal and R.K.\ Getoor (1960):
Some theorems on stable processes. \textit{Transactions of the American Mathematical Society} \textbf{95}, 263--273.


\bibitem{BSW}
B.\ B\"{o}ttcher, R. L.\ Schilling and J.\ Wang (2013):
\textit{L\'evy-Type Processes: Construction, Approximation and Sample Path Properties}.
Lecture Notes in Mathematics \textbf{2099}, L\'evy Matters III, Springer, Cham.

\bibitem{CMS76}
J.M.\ Chambers, C.L.\ Mallows and B.\ Stuck (1976):
A method for simulating stable random variables. \emph{Journal of the American Statistical Association} \textbf{71}, 340--344.

\bibitem{CNXY19} P.\ Chen, I.\ Nourdin, L.\ Xu and X.\ Yang (2019):
Multivariate stable approximation in Wasserstein distance by Stein's method. Preprint \textit{arXiv}: 1911.12917.


\bibitem{CLX21}
P.\ Chen, J.\ Lu and L.\ Xu (2022):
Approximation to stochastic variance reduced gradient Langevin dynamics by stochastic delay differential equations.
\emph{Applied Mathematics \& Optimization} \textbf{85}, 1--40.

\bibitem{CX19}
P.\ Chen and L.\ Xu (2019):
Approximation to stable law by the Lindeberg principle.
\textit{Journal of Mathematical Analysis and Applications} \textbf{480}, 123338.


\bibitem{DG20}
K.\ Dareiotis and M.\ Gerencs (2020):
On the regularisation of the noise for the Euler-Maruyama scheme with irregular drift.
\textit{Electronic Journal of Probability}  \textbf{25}, 1--18.

\bibitem{DSS17}
C.-S.\ Deng, R.L.\ Schilling and Y.-H.\ Song (2017):
Subgeometric rates of convergence for Markov processes under subordination.
\textit{Advances in Applied Probability} \textbf{49}, 162--181.

\bibitem{EG19}
A.\ Eberle and A.\ Guillin (2019):
Couplings and quantitative contraction rates for Langevin dynamics.
\textit{The Annals of Probability}  \textbf{47}, 1982--2010.

\bibitem{EGZ19}
A.\ Eberle, A.\ Guillin and R.\ Zimmer (2019):
Quantitative Harris-type theorems for diffusions and McKean-Vlasov processes.
\textit{Transactions of the American Mathematical Society}  \textbf{371}, 7135--7173.

\bibitem{FSX19}
X.\ Fang, Q.-M.\ Shao and L.\ Xu (2019):
Multivariate approximations in Wasserstein distance by Stein's method and Bismut's formula.
\textit{Probability Theory and Related Fields} \textbf{174}, 945--979.

\bibitem{FG16}
W.\ Fang and M.B.\ Giles (2016):
Adaptive Euler-Maruyama method for SDEs with non-globally Lipschitz drift.
In: \textit{International Conference on Monte Carlo and Quasi-Monte Carlo Methods in Scientific Computing}. Springer, Cham, pp.\ 217--234.

\bibitem{GW17}
Y.\ Guan and J.\ Wu (2017):
Exponential ergodicity for non-Lipschitz multivalued stochastic differential equations with L\'evy jumps.
\textit{Infinite Dimensional Analysis, Quantum Probability and Related Topics}  \textbf{20}, 1750002.

\bibitem{Hal81}
P.\ Hall (1981):
Two-sided bounds on the rate of convergence to a stable law.
\emph{Zeitschrift fur Wahrscheinlichkeitstheorie und Verwandte Gebiete} \textbf{57}, 349--364.

\bibitem{Jac04}
J.\ Jacod (2004):
The Euler scheme for L\'evy driven stochastic differential equations: limit theorems.
\textit{The Annals of Probability}  \textbf{32}, 1830--1872.

\bibitem{JPXX21}
X.\ Jin, G.\ Pang, L.\ Xu and X.\ Xu (2021):
An approximation to steady-state of $M/Ph/n+ M$ queue.
Preprint \textit{arXiv}:2109.03623.

\bibitem{JMW96}
A.\ Janicki, Z.\ Michna and A.\ Weron (1996):
Approximation of stochastic differential equations driven by $\alpha$-stable L\'evy motion.
\textit{Applicationes Mathematicae} \textbf{24}, 149--168.

\bibitem{KS19}
F.\ K\"{u}hn and R.L.\ Schilling (2019):
Strong convergence of the Euler-Maruyama approximation for a class of L\'evy-driven SDEs.
\textit{Stochastic Processes and their Applications}  \textbf{129}, 2654--2680.

\bibitem{LW20}
M.\ Liang and J.\ Wang (2020):
Gradient estimates and ergodicity for SDEs driven by multiplicative L\'evy noises via coupling.
\textit{Stochastic Processes and their Applications}  \textbf{130}, 3053--3094.

\bibitem{Lem07}
V.\ Lemaire (2007):
An adaptive scheme for the approximation of dissipative systems.
\textit{Stochastic Processes and their Applications}  \textbf{117}, 1491--1518.


\bibitem{LMYY18}
X.\ Li, Q.\ Ma, H.\ Yang and C.\ Yuan (2018):
The numerical invariant measure of stochastic differential equations with Markovian switching.
\textit{SIAM Journal on Numerical Analysis}  \textbf{56}, 1435--1455.

\bibitem{LTX20}
J.\ Lu, Y.\ Tan and L.\ Xu (2020):
Central limit theorem and self-normalized Cram\'er-type moderate deviation for Euler-Maruyama scheme.
\emph{Bernoulli} \textbf{28}, 937--964.

\bibitem{LW19}
D.\ Luo and J.\ Wang: (2019):
Refined basic couplings and Wasserstein-type distances for SDEs with L\'evy noises.
\textit{Stochastic Processes and their Applications}  \textbf{129}, 3129--3173.

\bibitem{M-T2}
S.P.\ Meyn and R.L.\ Tweedie (1992):
Stability of Markovian processes I: Criteria for discrete-time chains.
\textit{Advances in Applied Probability}  \textbf{24}, 542--574.

\bibitem{M-T3}
S.P.\ Meyn and R.L.\ Tweedie (1993):
Stability of Markovian processes III: Foster-Lyapunov criteria for continuous-time processes.
\textit{Advances in Applied Probability}  \textbf{25}, 518--548.

\bibitem{MX18}
R.\ Mikulevi\v{c}ius and F.\ Xu (2018):
On the rate of convergence of strong Euler approximation for SDEs driven by L\'evy processes.
\textit{Stochastics}  \textbf{90}, 569--604.

\bibitem{MN94}
R.\ Modarres and J.P.\ Nolan (1994):
A method for simulating stable random vectors.
\emph{Computational Statistics} \textbf{9}, 11--19.

\bibitem{NSR19}
T. H.\ Nguyen, U.\ Simsekli and G.\ Richard (2019):
Non-asymptotic analysis of Fractional Langevin Monte Carlo for non-convex optimization.
In: \textit{International Conference on Machine Learning}. Proceedings of Machine Learning Research, pp.\ 4810--4819.



\bibitem{Nol08}
J.P.\ Nolan (2008):
An overview of multivariate stable distributions.
Online: \url{http://hdl.handle.net/1961/auislandora:68717} (accessed January 12, 2023).

\bibitem{Nolbook}
J.P.\ Nolan (2020): \textit{Univariate stable distributions: models for heavy tailed data}. Springer Series in Operations
Research and Financial Engineering, Springer, Cham.




\bibitem{Nor86}
J.\ Norris (1986):
Simplified Malliavin calculus.
In: \textit{S\'{e}minaire de Probabilit\'{e}s XX 1984/85}.
Springer, Berlin, pp.\ 101--130.

\bibitem{Nua06} D.\ Nualart (2006): \textit{The Malliavin Calculus and Related Topics} (2nd ed). Springer, Berlin.

\bibitem{PP20}
G. Pag\`es and F.\ Panloup (2020):
Unadjusted Langevin algorithm with multiplicative noise: Total variation and Wasserstein bounds.
Preprint \textit{arXiv}:2012.14310.


\bibitem{Pan08}
F.\ Panloup (2008):
Recursive computation of the invariant measure of a stochastic differential equation driven by a L\'{e}vy process.
\textit{The Annals of Applied Probability}  \textbf{18}, 379--426.


\bibitem{PT17}
O.M.\ Pamen and D.\ Taguchi (2017):
Strong rate of convergence for the Euler-Maruyama approximation of SDEs with H\"older continuous drift coefficient.
\textit{Stochastic Processes and their Applications}  \textbf{127}, 2542--2559.

\bibitem{Pro04}
P.E.\ Protter (2004):
\textit{Stochastic Integration and Differential Equations} 2nd ed. Springer, Berlin.

\bibitem{PT97}
P.\ Protter and D.\ Talay (1997):
The Euler scheme for L\'evy driven stochastic differential equations.
\textit{The Annals of Probability}  \textbf{25}, 393--423.



\bibitem{SZ21}
J.M.\ Sanz-Serna and K.C.\ Zygalakis (2021):
Wasserstein distance estimates for the distributions of numerical approximations to ergodic stochastic differential equations.
\textit{Journal of Machine Learning Research}  \textbf{22}, 1--37.

\bibitem{sato}
K.\ Sato (1999):
\textit{L\'evy processes and infinitely divisible distributions}.
Cambridge University Press, Cambridge.

\bibitem{Sha18}
J.\ Shao (2018):
Weak convergence of Euler-Maruyama's approximation for SDEs under integrability condition.
Preprint \textit{arXiv}:1808.07250.

\bibitem{SSG19}
U.\ Simsekli, L.\ Sagun and M.\ Gurbuzbalaban (2019):
A tail-index analysis of stochastic gradient noise in deep neural networks.
In: \textit{International Conference on Machine Learning}. Proceedings of Machine Learning Research, pp. 5827--5837.

\bibitem{Tal90}
D.\ Talay (1990):
Second-order discretization schemes of stochastic differential systems for the computation of the invariant law.
\textit{Stochastics: An International Journal of Probability and Stochastic Processes}  \textbf{29}, 13--36.

\bibitem{TA18}
B.\ Tarami and M.\ Avaji (2018):
Convergence of Euler-Maruyama Method for Stochastic Differential Equations Driven by $\alpha$-stable L\'evy Motion.
\textit{Journal of Mathematical Extension} \textbf{12}, 31--53.

\bibitem{Wan20}
F. Y.\ Wang (2020):
Exponential contraction in Wasserstein distances for diffusion semigroups with negative curvature.
\textit{Potential Analysis}  \textbf{53}, 1123--1144.

\bibitem{W08}
J.\ Wang (2008):
Criteria for ergodicity of L\'evy type operators in dimension one.
\textit{Stochastic processes and their applications}  \textbf{118}, 1909-1928.

\bibitem{Wan16}
J.\ Wang (2016):
$L^{p}$-Wasserstein distance for stochastic differential equations driven by L\'evy processes.
\textit{Bernoulli}  \textbf{22}, 1598--1616.

\bibitem{Xu19}
L.\ Xu (2019):
Approximation of stable law in Wasserstein-1 distance by Stein's method.
\textit{Annals of Applied Probability} \textbf{29}, 458--504.


\bibitem{Z}
X.\ Zhang (2013):
Derivative formulas and gradient estimates for SDEs driven by $\alpha$-stable processes.
\textit{Stochastic Processes and their Applications}  \textbf{123}, 1213--1228.

\bibitem{ZFM20}
P.\ Zhou, J.\ Feng, C.\ Ma, C.\ Xiong, S.C.H.\ Hoi (2020):
Towards theoretically understanding why sgd generalizes better than adam in deep learning.
\textit{Advances in Neural Information Processing Systems} \textbf{33}, 21285--21296.

\end{thebibliography}

\end{document}